\documentclass[a4paper,10pt, oneside]{article}
\usepackage{amsmath,amssymb,amsthm,bm,mathrsfs,graphicx}
\usepackage{authblk}
\usepackage{enumerate}
\usepackage[font=small,labelfont=md,textfont=it]{caption}
\usepackage{floatrow,float}
\usepackage[titletoc, title]{appendix}
\usepackage[colorlinks,linkcolor=blue,citecolor=blue]{hyperref}
\usepackage{etoolbox}
\usepackage{longtable}
\usepackage{diagbox}
\usepackage{booktabs,makecell,multirow}
\usepackage[capitalise,nosort]{cleveref}
\usepackage{subeqnarray,cases,color}
\crefname{equation}{}{}
\crefname{lemma}{Lemma}{Lemmas}
\crefname{theorem}{Theorem}{Theorems}
\crefname{discr}{Discretization}{Discretizations}

\apptocmd{\sloppy}{\hbadness 10000\relax}{}{}

\newcommand{\nm}[1]{\lVert {#1} \rVert}
\newcommand{\nmB}[1]{\Big\Vert {#1} \Big\Vert}

\newcommand{\ssnm}[1]
{
  \left\vert\kern-0.25ex
  \left\vert\kern-0.25ex
  \left\vert
  {#1}
  \right\vert\kern-0.25ex
  \right\vert\kern-0.25ex
  \right\vert
}
\newcommand{\ssnmb}[1]
{
  \big\vert\kern-0.25ex
  \big\vert\kern-0.25ex
  \big\vert
  {#1}
  \big\vert\kern-0.25ex
  \big\vert\kern-0.25ex
  \big\vert
}
\newcommand{\ssnmB}[1]
{
  \Big\vert\kern-0.25ex
  \Big\vert\kern-0.25ex
  \Big\vert
  {#1}
  \Big\vert\kern-0.25ex
  \Big\vert\kern-0.25ex
  \Big\vert
}

\makeatletter
\def\spher@harm#1{%
  \vbox{\hbox{%
    \offinterlineskip
    \valign{&\hb@xt@2\p@{\hss$##$\hss}\vskip.2ex\cr#1\crcr}%
  }\vskip-.36ex}%
}
\def\gshone{\spher@harm{.}}
\def\gshtwo{\spher@harm{.&.}}
\def\gshthree{\spher@harm{.&.&.}}
\let\gsh\spher@harm
\makeatother

\newtheorem{hypothesis}{Hypothesis}[section]
\newtheorem{lemma}{Lemma}[section]
\newtheorem{remark}{Remark}[section]
\newtheorem{theorem}{Theorem}[section]

\makeatletter\def\@captype{table}\makeatother

\begin{document}


\title{Temporal semi-discretizations of a backward semilinear stochastic evolution equation
\thanks{This work was supported in part by the National Natural
Science Foundation of China (11901410, 11771312) and
the Fundamental Research Funds for the Central Universities
in China (2020SCU12063).}}

\author{Binjie Li\thanks{libinjie@scu.edu.cn}}
\author{Xiaoping Xie\thanks{Corresponding author: xpxie@scu.edu.cn}}
\affil{School of Mathematics, Sichuan University, Chengdu 610064, China}

\date{}
\maketitle

\begin{abstract} 
  This paper studies the convergence of three temporal semi-discretizations for
  a backward semilinear stochastic evolution equation. For general terminal
  value and general coefficient with Lipschitz continuity, the convergence of
  the first two temporal semi-discretizations is established, and an explicit
  convergence rate is derived for the third temporal semi-discretization. The
  third temporal semi-discretization is applied to a general stochastic linear
  quadratic control problem, and the convergence of a temporally semi-discrete
  approximation to the optimal control is established.
\end{abstract}



\section{Introduction}
In the literature, Bismut~\cite{Bismut1973} first introduced the finite
dimensional linear backward stochastic differential equations (BSDEs, for short)
to study the stochastic optimal control problems. Later, Pardoux and Peng
\cite{Pardoux_peng1990} studied the general finite dimensional BSDEs with
Lipschitz nonlinearity, and Hu and Peng \cite{Hu_Peng_1991} established the
well-posedness for the backward semilinear stochastic evolution equations with
Lipschitz nonlinearity. Since then a considerable number of papers have been
published for the applications of the BSDEs to stochastic optimal control,
partial differential equations and mathematical finance; see
\cite{Karoui1997,Ma_Yong1999,Pardoux2014,Peng1993,Yong1999} and the references cited therein.
We particularly refer the reader to \cite{Fuhrman2012,Fuhrman2013, Fuhrman2016,
Fuhrman2002,Fuhrman2004, Guatteri2011, Guatteri2013,Guatteri2005, Guatteri2014}
and the references  therein for the applications of the backward stochastic
partial differential equations to the stochastic optimal control problems.


By now, the numerical solutions of the finite-dimensional BSDEs have been
extensively studied. We particularly introduce several works as follows. For
backward-forward SDEs, Ma et al.~\cite{Ma1994} proposed a four-step scheme,
Zhang ~\cite{Zhang2004} and Bouchard and Touzi~\cite{Bouchard2004} analyzed two
Euler type schemes, and Chassagneux~\cite{Chassagneux2014} studied a class of
linear multistep methods. The above four works all require that the coefficients
are deterministic. For a class of nonlinear BSDEs with particular terminal value
and sufficiently smooth deterministic coefficients, Zhao et al.~\cite{Zhao2010}
proposed a stable multistep scheme. For the nonlinear BSDEs with general
terminal value and general coefficients, Hu et al.~\cite{Hu_Dualart_2011}
analyzed three schemes with some restrictions on the regularity of the
underlying solution, and these restrictions might be difficult to verify. We
also refer the reader to the references cited in the above papers for more
related works. Additionally, because of the close connections between the
stochastic evolution equations and the backward stochastic evolution equations,
we refer the reader to
\cite{Beccari2019,Cao2017,Cui_Hong_2019,Du2002,Hutzenthaler2015,
Hutzenthaler_Jentzen2020,Jentzen2009,Jentzen2015,Krusebook2014,Yan2005} and
the references   therein, for the numerical analysis of the stochastic
evolution equations.





Compared with the numerical analysis of the finite-dimensional BSDEs, the
numerical analysis of the backward stochastic semilinear evolution equations is
very limited. Wang~\cite{WangY2016} analyzed a discretization for a backward
semilinear stochastic parabolic equation; since this discretization uses the
eigenvectors of the Laplace operator, its application appears to be limited.
Recently, Li and Xie~\cite{LiXie2021} analyzed a spatial semi-discretization for
a backward semilinear stochastic parabolic equation with general filtration,
using the standard piecewise linear finite element method. To our best
knowledge, no numerical analysis of temporal semi-discretizations is available
for a backward semilinear stochastic evolution equation in an
infinite-dimensional Hilbert space.

The immaturity of the numerical analysis of the backward semilinear stochastic
evolution equations motivates us to study the temporal semi-discretizations for
the  equation
\begin{equation}
  \label{eq:model}
  \begin{cases}
    \mathrm{d}p(t) = -( A p(t) + f(t,p(t),z(t)) ) \, \mathrm{d}t +
    z(t) \, \mathrm{d}W(t), \quad 0 \leqslant t \leqslant T, \\
    p(T) = p_T,
  \end{cases}
\end{equation}
where $ 0 < T < \infty $, $ W(\cdot) $ is a one-dimensional real Brownian
motion, and $ p_T $ and $ f $ are given. One key difficulty in the numerical
analysis of the backward semilinear stochastic evolution equation
\cref{eq:model} is that the process $ z $ is generally of low temporal
regularity.
In this paper, we analyze three Euler type temporal semi-discretizations for
equation \cref{eq:model}. For the first two semi-discretizations, the process $
z $ is discretized by the piecewise constant processes, and we prove that the two
semi-discretizations are convergent. More precisely, we obtain the error bound
\[
  c(\tau^{1/2} + \ssnm{z - \mathcal P_\tau z}_{L^2(0,T;H)}),
\]
where $ \mathcal P_\tau z $ is the optimal approximation of $ z $ in the space
of piecewise constant processes.
Hence, if the process $ z $ indeed possesses higher temporal regularity,
then an explicit convergence rate will readily be derived. For the third
semi-discretization, the process $ z $ is not discretized, and an explicit
convergence rate is derived.
Finally, we apply the third semi-discretization to a general stochastic linear
quadratic control problem, and establish the convergence of a temporally
semi-discrete approximation, with reasonable regularity assumption on the data.
To sum up, our  main contributions lie in the following  aspects.
\begin{itemize}
  \item This work, to our best knowledge, provides the first numerical analysis of temporal semi-discretizations for an infinite-dimensional
    semilinear BSDE.

  \item Our   analysis, compared with most of the numerical analysis of the finite-dimensional BSDEs,  neither requires  the terminal value to be  generated by a forward stochastic evolution equation, nor requires the coefficient to be deterministic. In addition, it    requires only some reasonable regularity assumptions on the data, and imposes no regularity restriction on the solution.

  \item In the literature, the numerical analysis of the stochastic optimal
    control problems governed by the SPDEs is very limited; see
    \cite{Dunst2016,LiXie2021,Wang2020a,Wang2020b,ZhouLi2021}. Our analysis
    for the temporal semi-discretization of the general stochastic linear
    quadratic control problem, as far as we know, appears to be the first
    numerical analysis of such kinds of problems where the noise is
    multiplicative and the diffusion term contains the control variable.
\end{itemize}



The rest of this paper is organized as follows. \Cref{sec:pre} introduces some
preliminaries. \Cref{sec:semidiscretization} gives three temporal
semi-discretizations and their error estimates. \Cref{sec:application}
applies the third temporal semi-discretization to a stochastic linear
quadratic control problem. Finally, \Cref{sec:conclusion} concludes this
paper.

\section{Preliminaries}
\label{sec:pre}
Let $ (\Omega, \mathcal F, \mathbb P) $ be a given complete probability space,
on which a one-dimensional Brownian motion $ W(\cdot) $ is defined. Let $
\mathbb F := \{\mathcal F_t \mid t \geqslant 0\} $ be the filtration generated
by $ W(\cdot) $ and augmented by the $ \mathbb P $-null sets of $ \mathcal F $.
We use $ \mathbb E $ to denote the expectation and use $ \mathbb E_t $ to denote
the conditional expectation with respect to $ \mathcal F_t $ for each $ t
\geqslant 0 $. For any separable Hilbert space $ X $ with norm $ \nm{\cdot}_X $,
we write the Hilbert space $ L^2(\Omega, \mathcal F_T, \mathbb P; X) $ as $
L^2(\Omega;X) $, and use $ \ssnm{\cdot}_X $ to denote its norm. Moreover, define
\begin{align*} 
  L_\mathbb F^2(0,T;X) &:= \Big\{
    \varphi :[0,T] \times \Omega \to X
     \mid \varphi \text{ is $ \mathbb F
    $-progressively measurable } \\
     & \qquad\qquad\qquad\qquad\qquad\qquad
     \text{ and } \int_0^T \ssnm{\varphi(t)}_X^2 \, \mathrm{d}t
    < \infty
  \Big\},
\end{align*}
and let $ L_\mathbb F^2(\Omega;C([0,T];X)) $ be the space of all $ \mathbb F
$-progressively measurable processes $ \varphi $ with continuous paths in $ X $
such that
\[
  \ssnm{\varphi}_{C([0,T];X)} := \Big(
    \mathbb E \sup_{t \in [0,T]}
    \nm{\varphi(t)}_X^2
  \Big)^{1/2} < \infty.
\]
The space $ L_\mathbb F^2(\Omega;C([0,T];X)) $ is a Banach space with respect to
the above norm $ \ssnm{\cdot}_{C([0,T];X)} $.

Let $ H $ be a real separable Hilbert space with inner product $ (\cdot, \cdot)
_{H} $. Assume that
\[
  A: \mathrm{Domain}(A) \subset H \to H
\]
is a linear operator satisfying the following properties:
\begin{itemize}
  \item $ A $ is self-adjoint, i.e.,
    \[
      (Av, w)_{H}  = (v, Aw)_{H}  \quad \text{for all}
      \quad v, w \in \mathrm{Domain}(A);
    \]
  \item $ A $ is surjective, and there exists a positive constant $ \delta $ such that
    \[
      (-Av,v)_{H} \geqslant \delta \nm{v}_{H}^2
      \quad \text{for all} \quad v \in
      \mathrm{Domain}(A);
    \]
  \item $ \mathrm{Domain}(A) $ is dense in $ H $, and $ \mathrm{Domain}
    (A) $, equipped with the norm $ \nm{A \cdot}_{H} $, is compactly
    embedded into $ H $.
\end{itemize}
It is evident that $ A $ will generate an analytic contractive semigroup $
\{e^{tA} \mid t
\geqslant 0\} $ on $ H $. For each $ 0 \leqslant \gamma \leqslant 1 $, define
\[
  H^\gamma := \{
    (-A)^{-\gamma} v \mid  v \in H
  \}
\]
and endow this space with the norm
\[
  \nm{v}_{H^\gamma} := \nm{(-A)^\gamma v}_{H}
  \quad \forall v \in H^\gamma.
\]
In the sequel, we will use $ [\cdot,\cdot] $ to denote the usual inner product
of the Hilbert space $ L^2(\Omega; H) $.

For any two Banach spaces $ B_1 $ and $ B_2 $, $ \mathcal L(B_1,B_2) $ is the
space of all bounded linear operators from $ B_1 $ to $ B_2 $, and $ \mathcal
L(B_1, B_1) $ is abbreviated to $ \mathcal L(B_1) $. We denote by $ I $ the
identity mapping.

Finally, for the data $ f $ and $ p_T $ in equation \cref{eq:model}, we make the
following assumptions.

\begin{hypothesis} 
  \label{hypo:main}
  We assume that (i)-(iv) hold:
  \end{hypothesis}
  \begin{enumerate}[(i)]
    \item The functional $ f: [0,T] \times \Omega \times H \times
      H \to H $ satisfies that
      \[
        f(\cdot, v, w) \in L_\mathbb F^2(0, T; H)
        \quad \text{for all} \quad v, w \in H.
      \]
    \item There exists a positive constant $ C_L $ such that, $ \mathbb P
      $ almost surely for almost every $ t \in [0,T] $,
      \begin{equation} 
        \label{eq:f-Lips}
        \nm{f(t,p_1,z_1) - f(t,p_2,z_2)}_{H} \leqslant
        C_L (\nm{p_1-p_2}_{H} + \nm{z_1-z_2}_{H})
      \end{equation}
      for all $ p_1, p_2, z_1, z_2 \in H $.
    \item $ p_T \in L^2(\Omega; H^{1/2}) $.
  \end{enumerate}

  Under the above hypothesis, equation \cref{eq:model} admits a unique mild
  solution $ (p,z) $, and
  \begin{equation} 
    \label{eq:regu}
    (p,z) \in \big(
      L_\mathbb F^2(\Omega;C([0,T]; H^{1/2})) \cap
      L_\mathbb F^2(0,T;H^{1})
    \big) \times L_\mathbb F^2(0,T;H^{1/2}).
  \end{equation}
\begin{remark} 
  By \cite[Theorem 3.1]{Hu_Peng_1991}, \cite[Theorem 4.4]{Guatteri2005}, and the
  basic theory of BSDEs (see, e.g., \cite[Chapter 5]{Pardoux2014}), the proof of
  \cref{eq:regu} is straightforward; see also \cite[Theorem 4.10]{LuZhangbook2021}.
\end{remark}

\section{Three temporal semi-discretizations}
\label{sec:semidiscretization}
Let $ J $ be a positive integer and define $ t_j := j \tau $ for each
$ 0 \leqslant j \leqslant J $, where $ \tau := T/J $. Define
\begin{align*} 
  \mathcal X_\tau := \big\{
    V: [0,T] \times \Omega \to H
    &
    \mid
    V(t_j) \in L^2(\Omega, \mathcal F_{t_j}, \mathbb P; H)
    \text{ and $ V $ is constant} \\
    &
    \qquad \text{ on }  [t_j,t_{j+1}) \text{ for each $ 0 \leqslant j < J $}
  \big\}.
\end{align*}
For any $ V \in \mathcal X_\tau $, we denote $ V(t_j) $, $ 0 \leqslant j
\leqslant J $, by $ V_j $ for convenience. For each $ 0 \leqslant j < J $,
define $ \delta W_j := W(t_{j+1}) - W(t_j) $, and define $ \mathcal I_\tau^j:
L^2(\Omega; H) \to L^2(\Omega, \mathcal F_{t_j}, \mathbb P; H) $ by
\begin{equation} 
  \label{eq:Itauj-def}
  \mathcal I_\tau^j v := \frac1\tau
  \mathbb E_{t_j}(v \delta W_j) \quad
  \forall v \in L^2(\Omega;H).
\end{equation}
We also let $ \mathcal P_\tau $ be the $ L^2(\Omega;L^2(0,T;H)) $-orthogonal
projection onto $ \mathcal X_\tau $; more concretely, for any $ v \in L^2(\Omega;
L^2(0,T;H)) $,
\[ 
  (\mathcal P_\tau v)(t) := \frac1\tau
  \mathbb E_{t_j} \int_{t_j}^{t_{j+1}} v(s) \, \mathrm{d}s
\]
for all $ t_j \leqslant t < t_{j+1} $ with $ 0 \leqslant j < J $. In the rest of
this paper, $ c $ denotes a generic positive constant, independent of $ \tau $,
and its value may differ in different places.

Now we present three temporal semi-discretizations of equation \cref{eq:model}.
The first semi-discretization seeks $ (P, Z) \in \mathcal X_\tau \times \mathcal
X_\tau $ by
\begin{subequations}
\label{eq:discretization1}
\begin{numcases}{}
  P_J = p_T, \label{eq:P-Z-PJ} \\
  Z_j = \mathcal I_\tau^j \Big(
    P_{j+1} + \int_{t_j}^{t_{j+1}}
    f(t,P_{j+1},Z_j) \, \mathrm{d}t
  \Big), \quad 0 \leqslant j < J, \label{eq:P-Z-Z} \\
  P_j - \mathbb E_{t_j} P_{j+1} = \tau A P_j +
  \mathbb E_{t_j} \int_{t_j}^{t_{j+1}}
  f(t,P_{j+1},Z_j) \, \mathrm{d}t,
  \quad 0 \leqslant j < J.
  \label{eq:P-Z-P}
\end{numcases}
\end{subequations}
The second semi-discretization seeks $ (P, Z) \in \mathcal X_\tau \times
\mathcal X_\tau $ by
\begin{subequations}
  \label{eq:discretization2}
  \begin{numcases}{}
    P_J = p_T, \\ 
    Z_j = \mathcal I_\tau^j P_{j+1}, \quad 0 \leqslant j < J, \\ 
    P_j - \mathbb E_{t_j} P_{j+1} = \tau A P_j +
    \mathbb E_{t_j} \int_{t_j}^{t_{j+1}}  f(t,P_{j+1},Z_j) \, \mathrm{d}t,
    \,\,\, 0 \leqslant j < J.
  \end{numcases}
\end{subequations}
The third semi-discretization seeks $ (P, Z) \in \mathcal X_\tau \times
L_\mathbb F^2(0,T;H) $ by
\begin{subequations}
  \label{eq:discretization3}
  \begin{numcases}{}
    P_J = p_T, \\ 
    P_j \!-\! P_{j+1} \!=\! \tau A P_j \!+\!
    \int_{t_j}^{t_{j+1}} \! f(t,P_{j+1},Z(t)) \mathrm{d}t \!-\!
    \int_{t_j}^{t_{j+1}} \! Z(t) \, \mathrm{d}W(t),
    \,\,\, 0 \leqslant j < J.
  \end{numcases}
\end{subequations}

The main results of this section are the following three theorems.
\begin{theorem} 
  \label{thm:conv1}
  Assume that \cref{hypo:main} holds and $ \tau < 1/C_\text{L}^2 $. Let $ (p,z)
  $ and $ (P,Z) $ be the solutions of \cref{eq:model,eq:discretization1},
  respectively. Then
  \begin{equation}
    \label{eq:conv}
    \begin{aligned}
      & \max_{0 \leqslant j < J}
      \ssnm{p(t_j) - P_j}_{H} +
      \ssnm{z-Z}_{L^2(0,T;H)} \\
      \leqslant{} &
      c \big(
        \tau^{1/2} +
        \ssnm{(I - \mathcal P_\tau) z}_{L^2(0,T;H)}
      \big).
    \end{aligned}
  \end{equation}
\end{theorem}

\begin{theorem} 
  \label{thm:conv2}
  Assume that \cref{hypo:main} holds. Let $ (p,z) $ and $ (P,Z) $ be the
  solutions of \cref{eq:model,eq:discretization2}, respectively. Then the error
  estimate \cref{eq:conv} still holds.
\end{theorem}

\begin{theorem} 
  \label{thm:conv3}
  Let $ (p,z) $ and $ (P,Z) $ be the solutions of
  \cref{eq:model,eq:discretization3}, respectively. Then, under the conditions of
  \cref{thm:conv1}, we have
  \begin{equation}
    \label{eq:conv3}
    \max_{0 \leqslant j < J}
    \ssnm{p(t_j) - P_j}_{H} +
    \ssnm{z-Z}_{L^2(0,T;H)}
    \leqslant c \tau^{1/2}.
  \end{equation}
\end{theorem}


We only provide a complete proof of \cref{thm:conv1}, since the proofs of
\cref{thm:conv2,thm:conv3} are similar (see \cref{rem:tmp}). To this end,
we proceed as follows.

\subsection{Preliminary results}
\label{ssec:pre}
We present some standard estimates as follows. For any $ 0 < t \leqslant T $ and
$ 0 \leqslant \beta \leqslant \gamma \leqslant 1 $, we have (see, e.g.,
\cite[Theorem~6.13, Chapter~2]{Pazy1983})
\begin{align} 
  \nm{e^{tA}}_{\mathcal L(H^\beta, H^\gamma)}
  \leqslant ct^{\beta-\gamma},
  \label{eq:etA} \\
  \nm{I - e^{tA}}_{\mathcal L(H^\gamma, H^\beta)}
  \leqslant c t^{\gamma-\beta}.
  \label{eq:I-etA}
\end{align}
By \cite[Theorem 7.3]{Thomee2006} we have, for any $ 0 \leqslant \beta \leqslant
1 $,
\begin{equation} 
  \label{eq:I-tauA-conv}
  \nm{e^{m\tau A} - (I - \tau A)^{-m}}_{
    \mathcal L(H^\beta,H)
  } \leqslant c \tau^{\beta} m^{\beta-1}
  \quad \forall m > 0.
\end{equation}
For any $ v \in H^{1/2} $ and $ g \in L^2(0,T;H) $, we have the following
estimates:
\begin{equation} 
  \label{eq:auxi}
  \sum_{j=0}^{J-1} \nmB{
    w_j - e^{(T-t_j)A}v - \int_{t_j}^T e^{(s-t_j)A}g(t) \, \mathrm{d}t
  }_H^2 \leqslant c\tau \big( \nm{v}_{H^{1/2}}^2 + \nm{g}_{L^2(0,T;H)}^2 \big)
\end{equation}
and, for any $ 0 \leqslant j < J $,
\begin{equation} 
 \label{eq:auxi_stab}
 \max_{j \leqslant k < J} \nm{w_k}_{H^{1/2}} +
 \Big(
 \sum_{k=j}^{J-1} \tau \nm{w_k}_{H^{1}}^2
 \Big)^{1/2} \leqslant
 c\big(
 \nm{v}_{H^{1/2}} + \nm{g}_{L^2(t_j,T;H)}
 \big),
\end{equation}
where $ \{w_j\}_{j=0}^{J-1} $ is defined by
\[ 
  w_j := (I-\tau A)^{-(J-j)} v +
  \sum_{k=j}^{J-1} (I - \tau A)^{-(k-j+1)} \int_{t_k}^{t_{k+1}}
  g(t) \, \mathrm{d}t \quad \forall 0 \leqslant j < J.
\]
In addition, for any $ v \in L^2(\Omega;H) $ and $ w \in L^2(\Omega, \mathcal
F_{t_j}, \mathbb P; H) $ with $ 0 \leqslant j < J $, the following properties
are easily verified by \cref{eq:Itauj-def}:
\begin{align} 
  & \mathcal I_\tau^j w = 0 \quad \mathbb P \text{-a.s.},
  \label{eq:Itauj-pro-1} \\
  & \big( I - \delta W_j \mathcal I_\tau^j \big)( \delta W_j w ) = 0
  \quad \mathbb P \text{-a.s.},
  \label{eq:Itauj-pro-2} \\
  & \big[
    v - \delta W_j \mathcal I_\tau^j v, \,
    \delta W_j w
  \big] = 0, \label{eq:Itauj-pro-3} \\
  & \ssnmb{v - \delta W_j \mathcal I_\tau^j v}_{H}^2 +
  \ssnmb{\delta W_j \mathcal I_\tau^j v}_{H}^2 =
  \ssnm{v}_{H}^2. \label{eq:Itauj-pro-4}
\end{align}

\begin{remark}
  The estimates \cref{eq:auxi,eq:auxi_stab} are standard; see, e.g.,
  \cite[Chapter~12]{Thomee2006}.
\end{remark}

\subsection{Three temporal semi-discretizations of a backward linear stochastic evolution equation}
This subsection studies the convergence of three temporal semi-discretizations
for the following backward linear stochastic evolution equation:
\begin{subequations}
\label{eq:lin}
\begin{numcases}{}
  \mathrm{d} p(t) = -(A p + g)(t) \mathrm{d}t +
  z(t) \mathrm{d}W(t), \quad 0 \leqslant t \leqslant T, \\
  p(T) = p_T,
\end{numcases}
\end{subequations}
where $ g \in L_{\mathbb F}^2(0,T;H) $ and $ p_T \in L^2(\Omega;
H) $. The main results are the following three lemmas.
\begin{lemma} 
  \label{lem:lin-conv}
  Assume that $ (p,z) $ is the solution of \cref{eq:lin} with $ p_T
  \in L^2(\Omega; H^{1/2}) $ and $ g \in L_\mathbb F^2(0, T; H) $.
  Define $ (P,Z) \in \mathcal X_\tau \times \mathcal X_\tau $ by
  \begin{subequations} 
  \label{eq:P-Z-lin}
  \begin{numcases}{}
  P_J = p_T ,  \label{eq:P-Z-lin-0} \\
  Z_j = \mathcal I_\tau^j \Big(
  P_{j+1} + \int_{t_j}^{t_{j+1}} g(t) \, \mathrm{d}t
  \Big), \quad 0 \leqslant j < J,   \label{eq:P-Z-2} \\
  P_j - \mathbb E_{t_j} P_{j+1} =
  \tau A P_j + \mathbb E_{t_j}
  \int_{t_j}^{t_{j+1}} g(t) \, \mathrm{d}t,
  \quad 0 \leqslant j < J. \label{eq:P-Z-1}
  \end{numcases}
  \end{subequations}
  Then the following estimates hold: for any $ 0 \leqslant j < J $,
  \begin{small}
  \begin{align}
    & \ssnm{p(t_j) - P_j}_{H}
    \leqslant c \tau^{1/2} \Big(
      (J-j)^{-1/2} \ssnm{p_T}_{H^{1/2}} +
      \ssnm{g}_{L^2(0,T;H)}
    \Big); \label{eq:lin-conv-inf} \\
    & \Big(
      \sum_{j=0}^{J-1} \ssnm{p - P_{j+1}}_{L^2(t_j,t_{j+1};H)}^2
    \Big)^{1/2} \leqslant
    c \tau^{1/2} \Big(
      \ssnm{p_T}_{H^{1/2}} + \ssnm{g}_{L^2(0,T;H))}
    \Big); \label{eq:lin-conv-p-Pj+1} \\
    & \ssnm{z-Z}_{L^2(0,T;H)}  \leqslant
    c\tau^{1/2} \Big(
      \ssnm{p_T}_{H^{1/2}} + \ssnm{g}_{L^2(0,T;H)}
    \Big) + \ssnm{(I-\mathcal P_\tau)z}_{L^2(0,T;H)}.
    \label{eq:lin-conv-z-Z}
  \end{align}
  \end{small}
\end{lemma}

\begin{lemma} 
  \label{lem:lin-conv2}
  Define $ (P,Z) \in \mathcal X_\tau \times
  \mathcal X_\tau $ by
  \begin{subequations}
  \begin{numcases}{}
    P_J = p_T, \notag \\
    Z_j = \mathcal I_\tau^j P_{j+1}, \qquad 0 \leqslant j < J, \notag \\
    P_j - \mathbb E_{t_j} P_{j+1} =
    \tau A P_j + \mathbb E_{t_j}
    \int_{t_j}^{t_{j+1}} g(t) \, \mathrm{d}t,
    \qquad 0 \leqslant j < J. \notag
  \end{numcases}
  \end{subequations}
 Then the three estimates in \cref{lem:lin-conv} still hold, under the
 conditions of \cref{lem:lin-conv}.
\end{lemma}

\begin{lemma} 
  \label{lem:lin-conv3}
 Define $ (P,Z) \in \mathcal X_\tau \times L_\mathbb F^2(0,T;H) $ by
 \begin{subequations}
  \begin{numcases}{}
    P_J = p_T, \notag \\
    P_j - P_{j+1} = \tau A P_j +
    \int_{t_j}^{t_{j+1}} g(t) \, \mathrm{d}t -
    \int_{t_j}^{t_{j+1}} Z(t) \, \mathrm{d}W(t),
    \qquad 0 \leqslant j < J. \notag
  \end{numcases}
  \end{subequations}
  Then, under the conditions of \cref{lem:lin-conv}, the error estimates
  \cref{eq:lin-conv-inf,eq:lin-conv-p-Pj+1} in \cref{lem:lin-conv} still hold,
  and
  \begin{equation}
    \label{eq:lin-conv3}
    \ssnm{z-Z}_{L^2(0,T;H)} \leqslant
    c \tau^{1/2} \big(
      \ssnm{p_T}_{H^{1/2}} +
      \ssnm{g}_{L^2(0,T;H)}
    \big).
  \end{equation}
\end{lemma}


Since the proofs of \cref{lem:lin-conv2,lem:lin-conv3} are similar to (and
simpler than) that of \cref{lem:lin-conv}, we only prove the latter. To this
end, we first present some standard properties of the solution $ (p,z) $ to
equation \cref{eq:lin} as follows:
\begin{itemize} 
  \item for any $ 0 \leqslant t \leqslant T $, we have
    \begin{align} 
      p(t) = \mathbb E_t \Big(
        e^{(T-t)A} p_T + \int_t^T e^{(r-t)A} g(r) \, \mathrm{d}r
      \Big) \quad \mathbb P \text{-a.s.};
      \label{eq:lin-mild-form}
    \end{align}
  \item for any $ 0 \leqslant s \leqslant t \leqslant T $, we have
    \begin{small}
      \begin{align} 
         & p(s) - p(t) = \int_s^t (Ap + g)(r) \, \mathrm{d}r -
         \int_s^t z(r) \, \mathrm{d}W(r) \quad \mathbb P \text{-a.s.},
         \label{eq:lin-strong} \\
         & p(s) - e^{(t-s)A} p(t) =
         \int_s^t e^{(r-s)A} g(r) \, \mathrm{d}r -
         \int_s^t e^{(r-s)A} z(r) \, \mathrm{d}W(r)
         \quad \mathbb P \text{-a.s.};
         \label{eq:p-z-int}
      \end{align}
    \end{small}
  \item for $ p_T \in L^2(\Omega;H^{1/2}) $ and $ g \in L_\mathbb F^2(0,T;H) $,
    we have
    \begin{equation}
      \label{eq:lin-p-z-regu}
      \ssnm{p}_{L^2(0,T;H^1)} + \ssnm{z}_{L^2(0,T;H^{1/2})}
      \leqslant c \big( \ssnm{p_T}_{H^{1/2}} + \ssnm{g}_{L^2(0,T;H)} \big).
    \end{equation}
\end{itemize}

\begin{remark} 
  The above properties are standard and easily verified by the Galerkin method and the basic properties of the finite-dimensional BSDEs (see, e.g.,
  \cite[Chapter~5]{Pardoux2014}).
\end{remark}


Then we present two technical lemmas, which can be proved by straightforward
calculations.

\begin{lemma}
  \label{lem:auxi-1}
  For any $ 0 \leqslant j < J $,
  \begin{equation} 
    \label{eq:auxi-1}
    \sum_{k=j}^{J-1} \int_{t_k}^{t_{k+1}}
    \nm{
      e^{(t-t_j)A} - (I - \tau A)^{-(k-j+1)}
    }_{\mathcal L(H)}^2 \, \mathrm{d}t
    \leqslant c \tau.
  \end{equation}
\end{lemma}
\begin{proof}
  We have
  \begin{align*} 
    & \sum_{k=j}^{J-1} \int_{t_k}^{t_{k+1}}
    \nm{
      e^{(t-t_j)A} - (I - \tau A)^{-(k-j+1)}
    }_{\mathcal L(H)}^2 \, \mathrm{d}t \notag \\
    ={}&
    \sum_{k=j}^{J-1} \int_{t_k}^{t_{k+1}}
    \nm{
      e^{(t-t_j)A} - e^{(t_{k+1} - t_j)A} +
      e^{(t_{k+1} - t_j)A} - (I - \tau A)^{-(k-j+1)}
    }_{\mathcal L(H)}^2 \, \mathrm{d}t \\
    \leqslant{} &
    2\sum_{k=j}^{J-1} \int_{t_k}^{t_{k+1}}
    \nm{e^{(t-t_j)A} - e^{(t_{k+1}-t_j)A}}_{\mathcal L(H)}^2 \,
    \mathrm{d}t + {} \\
    &
    \quad 2\sum_{k=j}^{J-1} \int_{t_k}^{t_{k+1}}
    \nm{e^{(t_{k+1}-t_j)A} - (I - \tau A)^{-(k-j+1)}}_{\mathcal L(H)}^2 \,
    \mathrm{d}t \\
    =: &
    \mathbb I_1 + \mathbb I_2.
  \end{align*}
  For $ \mathbb I_1 $ we have
  \begin{align*} 
    \mathbb I_1 & =
    2 \sum_{k=j}^{J-1} \int_{t_k}^{t_{k+1}}
    \nm{
      (I - e^{(t_{k+1} - t)A}) e^{(t-t_j)A}
    }_{\mathcal L(H)}^2 \, \mathrm{d}t \\
    & \leqslant
    2\int_{t_j}^{t_{j+1}}
    \nm{
      I - e^{(t_{j+1} - t)A }
    }_{\mathcal L(H)}^2
    \nm{ e^{(t-t_j)A} }_{\mathcal L(H)}^2
    \, \mathrm{d}t \\
    & \quad {} + 2\sum_{k=j+1}^{J-1} \int_{t_k}^{t_{k+1}} \nm{
      I - e^{(t_{k+1} - t)A }
    }_{\mathcal L(H^1, H)}^2
    \nm{ e^{(t-t_j)A} }_{\mathcal L(H,H^1)}^2
    \, \mathrm{d}t \\
    & \leqslant c \tau,
  \end{align*}
  by the following two estimates:
  \begin{align*}
    & \int_{t_j}^{t_{j+1}}
    \nm{
      I - e^{(t_{j+1} - t)A }
    }_{\mathcal L(H)}^2
    \nm{ e^{(t-t_j)A} }_{\mathcal L(H)}^2
    \, \mathrm{d}t \\
    \leqslant{} &
    c \int_{t_j}^{t_{j+1}} \, \mathrm{d}t
    \quad \text{(by \cref{eq:etA,eq:I-etA})} \\
    \leqslant{} &
    c \tau
  \end{align*}
  and
  \begin{align*}
    & \sum_{k=j+1}^{J-1} \int_{t_k}^{t_{k+1}}
    \nm{
      I- e^{(t_{k+1} - t)A}
    }_{\mathcal L(H^1, H)}^2
    \nm{ e^{(t-t_j)A} }_{\mathcal L(H, H^1)}^2 \, \mathrm{d}t \\
    \leqslant{} &
    c \sum_{k=j+1}^{J-1} \int_{t_k}^{t_{k+1}}
    (t_{k+1} - t)^2 (t-t_j)^{-2} \, \mathrm{d}t
    \quad \text{(by \cref{eq:etA,eq:I-etA})} \\
    \leqslant{} &
    c \tau^2 \sum_{k=j+1}^{J-1} \int_{t_k}^{t_{k+1}}
    (t-t_j)^{-2} \, \mathrm{d}t \\
    \leqslant{} &
    c \tau.
  \end{align*}
  For $ \mathbb I_2 $, by \cref{eq:I-tauA-conv} we obtain
  \[
    \mathbb I_2 \leqslant
    c \sum_{k=j}^{J-1} \tau (k-j+1)^{-2}
    \leqslant c \tau.
  \]
  Combining the above estimates of $ \mathbb I_1 $ and $ \mathbb I_2 $ yields
  \cref{eq:auxi-1} and thus completes the proof.
\end{proof}

\begin{lemma}
  \label{lem:Ap+g}
  Let $ (p,z) $ be the solution to equation \cref{eq:lin} with $ g \in L_\mathbb
  F^2(0,T; H) $ and $ p_T \in L^2(\Omega;H^{1/2}) $. Then
  \begin{equation} 
    \label{eq:Ap+g}
    \sum_{k=0}^{J-1} \Big(
    \int_{t_k}^{t_{k+1}} \ssnm{p(t)}_{H^1} \, \mathrm{d}t
    \Big)^2 \leqslant
    c \tau \big( \ssnm{p_T}_{H^{1/2}}^2 +
    \ssnm{g}_{L^2(0,T;H)}^2 \big).
  \end{equation}
\end{lemma}
\begin{proof}
  Let
  \[
    \eta(t) := e^{(T-t)A} p_T + \int_t^T e^{(s-t)A} g(s) \, \mathrm{d}s
    \quad \forall 0 \leqslant t \leqslant T.
  \]
  It is standard that
  \begin{equation} 
    \label{eq:eta}
    \ssnm{\eta}_{L^2(0,T;H^1)}
    \leqslant c\big( \ssnm{p_T}_{H^{1/2}} + \ssnm{g}_{L^2(0,T;H)} \big).
  \end{equation}
  By \cref{eq:lin-mild-form} we have, for any $ 0 \leqslant t < T $,
  \begin{align*}
    \ssnm{p(t)}_{H^1} &=
    \ssnm{ \mathbb E_t \eta(t) }_{H^1} \leqslant
    \ssnm{ \eta(t) }_{H^1},
  \end{align*}
  so that
  \begin{align*} 
    & \sum_{k=0}^{J-1} \Big(
      \int_{t_k}^{t_{k+1}} \ssnm{p(t)}_{H^1} \, \mathrm{d}t
    \Big)^2 \leqslant
    \sum_{k=0}^{J-1} \tau  \int_{t_k}^{t_{k+1}}
    \ssnm{p(t)}_{H^1}^2 \, \mathrm{d}t \\
    \leqslant{} &
    \sum_{k=0}^{J-1} \tau \int_{t_k}^{t_{k+1}}
    \ssnm{\eta(t)}_{H^1}^2 \, \mathrm{d}t =
    \tau\ssnm{\eta}_{L^2(0,T;H^1)}^2,
  \end{align*}
  which, together with \cref{eq:eta}, proves the desired estimate \cref{eq:Ap+g}.
\end{proof}

Finally, we are in a position to prove \cref{lem:lin-conv} as follows.

\medskip\noindent{\bf Proof of \cref{lem:lin-conv}}. Firstly, let us prove
  \cref{eq:lin-conv-inf}. Let $ 0 \leqslant j < J $ be arbitrary but fixed. From
  \cref{eq:P-Z-lin}, it is easily verified that
  \begin{equation}
    \label{eq:xy}
    P_j = \mathbb E_{t_j} \Big(
      (I-\tau A)^{-(J-j)} p_T + \sum_{k=j}^{J-1}
      \int_{t_k}^{t_{k+1}} (I - \tau A)^{-(k-j+1)} g(t) \, \mathrm{d}t
    \Big) \quad \mathbb P \text{-a.s.}
  \end{equation}
  Hence, by \cref{eq:lin-mild-form} we obtain
  \[ 
    p(t_j) - P_j = \mathbb I_1 + \mathbb I_2
    \quad \mathbb P \text{-a.s.,}
  \]
  where
  \begin{align*}
    \mathbb I_1 &:= \mathbb E_{t_j} \big(
      e^{(T-t_j)A} - (I - \tau A)^{-(J-j)}
    \big) p_T, \\
    \mathbb I_2 &:= \mathbb E_{t_j} \Big(
      \sum_{k=j}^{J-1} \int_{t_k}^{t_{k+1}}
      \big( e^{(t-t_j)A} - (I - \tau A)^{-(k-j+1)} \big)
      g(t) \, \mathrm{d}t
    \Big).
  \end{align*}
  For $ \mathbb I_1 $ we have
  \begin{align*} 
    \ssnm{\mathbb I_1}_{H} &=
    \ssnm{
      \mathbb E_{t_j} \Big(
        e^{(T-t_j)A} - (I - \tau A)^{-(J-j)}
      \Big) p_T
    }_{H} \\
    & \leqslant
    \ssnm{
      \big(
        e^{(T-t_j)A} - (I - \tau A)^{-(J-j)}
      \big) p_T
    }_{H} \\
    & \leqslant
    \nm{
      e^{(T-t_j)A} - (I - \tau A)^{-(J-j)}
    }_{\mathcal L(H^{1/2}, H)} \ssnm{p_T}_{H^{1/2}} \\
    & \leqslant c \tau^{1/2}
    (J-j)^{-1/2} \ssnm{p_T}_{H^{1/2}}
    \quad\text{(by \cref{eq:I-tauA-conv}).}
  \end{align*}
  For $ \mathbb I_2 $ we have
  \begin{small}
  \begin{align*}
    \ssnm{\mathbb I_2}_{H} &=
    \ssnmB{
      \mathbb E_{t_j} \Big(
        \sum_{k=j}^{J-1} \int_{t_k}^{t_{k+1}}
        \big(
          e^{(t-t_j)A} - (I - \tau A)^{-(k-j+1)}
        \big) g(t) \, \mathrm{d}t
      \Big)
    }_{H} \\
    & \leqslant
    \ssnmB{
      \sum_{k=j}^{J-1} \int_{t_k}^{t_{k+1}}
      \big(
        e^{(t-t_j)A} - (I - \tau A)^{-(k-j+1)}
      \big) g(t) \, \mathrm{d}t
    }_{H} \\
    &\leqslant
    \sum_{k=j}^{J-1} \int_{t_k}^{t_{k+1}}
    \ssnmb{
      ( e^{(t-t_j)A} - (I-\tau A)^{-(k-j+1)} )g(t)
    }_{H} \, \mathrm{d}t \\
    &\leqslant
    \sum_{k=j}^{J-1} \int_{t_k}^{t_{k+1}}
    \nm{
      e^{(t-t_j)A} - (I-\tau A)^{-(k-j+1)}
    }_{\mathcal L(H)}
    \ssnm{g(t)}_{H} \, \mathrm{d}t \\
    & \leqslant
    \Big(
      \sum_{k=j}^{J-1} \int_{t_k}^{t_{k+1}}
      \nm{
        e^{(t-t_j)A} - (I - \tau A)^{-(k-j+1)}
      }_{\mathcal L(H)}^2 \, \mathrm{d}t
    \Big)^{1/2} \ssnm{g}_{L^2(t_j,T;H)} \\
    & \leqslant c \tau^{1/2}
    \ssnm{g}_{L^2(t_j,T;H)}
    \quad \text{(by \cref{lem:auxi-1})}.
  \end{align*}
  \end{small}
  Combining the above estimates of $ \mathbb I_1 $ and $ \mathbb I_2 $ then
  yields \cref{eq:lin-conv-inf}.

  Secondly, let us prove \cref{eq:lin-conv-p-Pj+1}. For any $ 0 \leqslant j < J
  $, by \cref{eq:lin-strong} we have
  \begin{align*}
    p(t) - p(t_{j+1}) = \int_t^{t_{j+1}}
    (Ap + g)(s) \, \mathrm{d}s - \int_t^{t_{j+1}}
    z(s) \, \mathrm{d}W(s), \, t_j \leqslant t < t_{j+1},
  \end{align*}
  and so
  \begin{align*}
    & \ssnm{p - p(t_{j+1})}_{L^2(t_j,t_{j+1};H)}^2 \\
    \leqslant{} &
    2 \int_{t_j}^{t_{j+1}} \ssnm{
      \int_t^{t_{j+1}} (Ap + g)(s) \, \mathrm{d}s
    }_H^2 \, \mathrm{d}t + 2 \int_{t_j}^{t_{j+1}} \ssnm{
      \int_t^{t_{j+1}} z(s) \, \mathrm{d}W(s)
    }_H^2 \, \mathrm{d}t \\
    ={} &
    2 \int_{t_j}^{t_{j+1}} \ssnm{
      \int_t^{t_{j+1}} (Ap + g)(s) \, \mathrm{d}s
    }_H^2 \, \mathrm{d}t + 2 \int_{t_j}^{t_{j+1}}
    \int_t^{t_{j+1}} \ssnm{z(s)}_H^2 \, \mathrm{d}s \, \mathrm{d}t \\
    \leqslant{} &
    2\tau^2 \ssnm{Ap + g}_{L^2(t_j,t_{j+1};H)}^2 +
    2\tau \ssnm{z}_{L^2(t_j,t_{j+1};H)}^2,
  \end{align*}
  It follows that
  \begin{align*}
    \sum_{j=0}^{J-1} \ssnm{p - p(t_{j+1})}_{L^2(t_j,t_{j+1};H)}^2
    & \leqslant
    2\tau^2 \ssnm{Ap + g}_{L^2(0,T;H)}^2 +
    2\tau \ssnm{z}_{L^2(0,T;H)}^2,
  \end{align*}
  and hence by \cref{eq:lin-p-z-regu} we obtain
  \begin{equation}
    \label{eq:lin2-p-pj}
    \sum_{j=0}^{J-1}
    \int_{t_j}^{t_{j+1}} \ssnm{p(t) - p(t_{j+1})}_{H}^2
    \, \mathrm{d}t
    \leqslant
    c \tau \big(
      \ssnm{p_T}_{H^{1/2}}^2 +
      \ssnm{g}_{L^2(0,T;H)}^2
    \big).
  \end{equation}
  Using the above estimate and \cref{eq:lin-conv-inf} yields
  \begin{align*}
    & \sum_{j=0}^{J-1} \ssnm{p - P_{j+1}}_{L^2(t_j,t_{j+1};H)}^2 \\
    ={} &
    \sum_{j=0}^{J-1} \ssnm{
      p - p(t_{j+1}) + p(t_{j+1}) - P_{j+1}
    }_{L^2(t_j,t_{j+1};H)}^2 \\
    \leqslant{} &
    2 \sum_{j=0}^{J-1} \ssnm{p - p(t_{j+1})}_{L^2(t_j,t_{j+1};H)}^2 +
    2 \sum_{j=0}^{J-1} \ssnm{p(t_{j+1}) - P_{j+1}}_{L^2(t_j,t_{j+1};H)}^2 \\
    \leqslant{} &
    c \tau \big(
      \ssnm{p_T}_{H^{1/2}}^2 +
      \ssnm{g}_{L^2(0,T;H)}^2
    \big),
  \end{align*}
  which implies the desired estimate \cref{eq:lin-conv-p-Pj+1}.

 Thirdly, let us prove \cref{eq:lin-conv-z-Z}. Fix $ 0 \leqslant j < J $. By
 \cref{eq:lin-strong} we have
 \begin{equation} 
   \label{eq:30000}
    p(t_{j+1}) + \int_{t_j}^{t_{j+1}} g(t) \, \mathrm{d}t =
    p(t_{j}) - \int_{t_j}^{t_{j+1}} Ap(t) \, \mathrm{d}t  +
    \int_{t_j}^{t_{j+1}} z(t) \, \mathrm{d}W(t)
    \quad \mathbb P \text{-a.s.,}
 \end{equation}
  so that, $ \mathbb P $-a.s.,
  \begin{small}
  \begin{align*} 
    & (I - \mathbb E_{t_j}-\delta W_j \mathcal I_\tau^j)
    \Big(
      p(t_{j+1}) + \int_{t_j}^{t_{j+1}} g(t) \, \mathrm{d}t
    \Big) \\
    ={} &
    (I - \mathbb E_{t_j}-\delta W_j \mathcal I_\tau^j)
    \Big(
      p(t_j) - \int_{t_j}^{t_{j+1}} Ap(t) \, \mathrm{d}t +
      \int_{t_j}^{t_{j+1}} z(t) \, \mathrm{d}W(t)
    \Big) \\
    ={} &
    -\delta W_j \mathcal I_\tau^j p(t_j) -
    (I - \mathbb E_{t_j} - \delta W_j \mathcal I_\tau^j)
    \int_{t_j}^{t_{j+1}} Ap(t) \, \mathrm{d}t +
    (I - \delta W_j \mathcal I_\tau^j)
    \int_{t_j}^{t_{j+1}} z(t) \, \mathrm{d}W(t) \\
    ={} &
    -(I - \mathbb E_{t_j} - \delta W_j \mathcal I_\tau^j)
    \int_{t_j}^{t_{j+1}} Ap(t) \, \mathrm{d}t +
    (I - \delta W_j \mathcal I_\tau^j)
    \int_{t_j}^{t_{j+1}} z(t) \, \mathrm{d}W(t)
    \quad \text{(by \cref{eq:Itauj-pro-1})}.
  \end{align*}
  \end{small}
  It follows that
  \begin{align*} 
    & \big(
      I - \mathbb E_{t_j} - \delta W_j \mathcal I_\tau^j
    \big) \Big(
      p(t_{j+1}) + \int_{t_j}^{t_{j+1}} g(t) \, \mathrm{d}t
    \Big) + (I - \mathbb E_{t_j}) \int_{t_j}^{t_{j+1}}
    Ap(t) \, \mathrm{d}t \\
    ={} &
    \delta W_j \mathcal I_\tau^j \int_{t_j}^{t_{j+1}} Ap(t) \, \mathrm{d}t
    + (I - \delta W_j \mathcal I_\tau^j) \int_{t_j}^{t_{j+1}}
    z(t) \, \mathrm{d}W(t) \quad \mathbb P \text{-a.s.},
  \end{align*}
  which further implies
  \begin{small}
  \begin{equation} 
    \label{eq:40000}
    \begin{aligned}
      & (I - \mathbb E_{t_j}) \Big(
        p(t_{j+1}) + \int_{t_j}^{t_{j+1}} (Ap + g)(t) \, \mathrm{d}t
      \Big) \\
      ={} &
      \delta W_j \mathcal I_\tau^j \Big(
        p(t_{j+1}) + \int_{t_j}^{t_{j+1}} (Ap + g)(t) \, \mathrm{d}t
      \Big) + (I - \delta W_j \mathcal I_\tau^j)
      \int_{t_j}^{t_{j+1}} z(t) \, \mathrm{d}W(t)
      \quad \mathbb P \text{-a.s.}
    \end{aligned}
  \end{equation}
  \end{small}
  By \cref{eq:30000} we also have
  \begin{align*} 
    \int_{t_j}^{t_{j+1}} z(t) \, \mathrm{d}W(t) &=
    (I - \mathbb E_{t_j}) \Big(
      p(t_{j+1}) - p(t_j) + \int_{t_j}^{t_{j+1}}
      (Ap+g)(t) \, \mathrm{d}t
    \Big) \\
    &= (I - \mathbb E_{t_j}) \Big(
      p(t_{j+1}) + \int_{t_j}^{t_{j+1}}
      (Ap+g)(t) \, \mathrm{d}t
    \Big) \quad \mathbb P \text{-a.s.,}
  \end{align*}
  which, together with \eqref{eq:P-Z-2} and  \cref{eq:Itauj-pro-2}, implies $ \mathbb P $-a.s.
  \begin{align*}
    & \int_{t_j}^{t_{j+1}} (z-Z)(t) \, \mathrm{d}W(t) \\
    ={} &
    (I - \mathbb E_{t_j}) \Big(
      p(t_{j+1}) + \int_{t_j}^{t_{j+1}} (Ap + g)(t) \, \mathrm{d}t
    \Big) -
    \delta W_j \mathcal I_\tau^j \Big(
      P_{j+1} + \int_{t_j}^{t_{j+1}} g(t) \, \mathrm{d}t
    \Big) \\
    ={} &
    \delta W_j \mathcal I_\tau^j \Big(
      p(t_{j+1}) \!-\! P_{j+1} \!+\! \int_{t_j}^{t_{j+1}} \! Ap(t) \mathrm{d}t
    \Big) + (I \!-\! \delta W_j \mathcal I_\tau^j) \!
    \int_{t_j}^{t_{j+1}} \!
    z(t) \mathrm{d}W(t) \quad \text{(by \cref{eq:40000})} \\
    ={} &
    \delta W_j \mathcal I_\tau^j \Big(
      p(t_{j+1}) \!-\! P_{j+1} \!+\!
      \int_{t_j}^{t_{j+1}} Ap(t) \mathrm{d}t
    \Big) \!+\! (I \!-\! \delta W_j \mathcal I_\tau^j)
    \int_{t_j}^{t_{j+1}}
    (z \!-\! \mathcal P_\tau z)(t) \mathrm{d}W(t).
  \end{align*}
   Hence,
  \begin{small}
  \begin{align*}
    & \ssnmB{
      \int_{t_j}^{t_{j+1}} (z-Z)(t) \, \mathrm{d}W(t)
    }_H^2 \\
    ={} & \ssnmB{
      \delta W_j \mathcal I_\tau^j \Big(
        p(t_{j+1}) \!-\! P_{j+1} \!+\!
        \int_{t_j}^{t_{j+1}} \!\! Ap(t) \mathrm{d}t \!
      \Big) \!
    }_H^2 \!+\! \ssnmB{
      (I \!-\! \delta W_j \mathcal I_\tau^j)
      \int_{t_j}^{t_{j+1}} \!\! (z \!-\! \mathcal P_\tau z)(t) \mathrm{d}W(t)
    }_H^2  \\ 
    \leqslant{} & \ssnmB{
        p(t_{j+1}) - P_{j+1} + \int_{t_j}^{t_{j+1}} Ap(t) \, \mathrm{d}t
    }_H^2 + \ssnmB{
      \int_{t_j}^{t_{j+1}} (z-\mathcal P_\tau z)(t) \, \mathrm{d}W(t)
    }_H^2 \quad \text{(by \cref{eq:Itauj-pro-4})} \\
    ={} & \ssnmB{
        p(t_{j+1}) - P_{j+1} + \int_{t_j}^{t_{j+1}} Ap(t) \, \mathrm{d}t
      }_H^2 + \ssnm{(I - \mathcal P_\tau)z}_{L^2(t_j,t_{j+1};H)}^2 \\
    \leqslant{} &
    2 \ssnm{p(t_{j+1}) - P_{j+1}}_H^2 +
    2 \ssnmB{ \int_{t_j}^{t_{j+1}} Ap(t) \, \mathrm{d}t }_H^2 +
    \ssnm{(I - \mathcal P_\tau) z}_{L^2(t_j,t_{j+1};H)}^2,
  \end{align*}
  \end{small}
  \hskip -4pt where we have used the property \cref{eq:Itauj-pro-3} in the first
  equality. Since $ 0 \leqslant j < J $ is arbitrary, summing over $ j $ from $
  0 $ to $ J-1 $ leads to
  \begin{small}
  \begin{align*} 
    & \sum_{j=0}^{J-1} \ssnmB{
      \int_{t_j}^{t_{j+1}} (z-Z)(t) \, \mathrm{d}W(t)
    }_H^2 \\
    \leqslant{} &
    2 \sum_{j=0}^{J-1} \ssnm{p(t_{j+1}) - P_{j+1}}_H^2 +
    2 \sum_{j=0}^{J-1} \ssnmB{ \int_{t_j}^{t_{j+1}} Ap(t) \, \mathrm{d}t }_H^2
    + \ssnm{(I - \mathcal P_\tau)z}_{L^2(0,T;H)}^2 \\
    \leqslant{} &
    2 \sum_{j=0}^{J-1} \ssnm{p(t_{j+1}) - P_{j+1}}_H^2 +
    2 \sum_{j=0}^{J-1} \Big(
      \int_{t_j}^{t_{j+1}} \ssnm{Ap(t)}_H \, \mathrm{d}t
    \Big)^2 + \ssnm{(I - \mathcal P_\tau)z}_{L^2(0,T;H)}^2 \\
    ={} &
    2 \sum_{j=0}^{J-1} \ssnm{p(t_{j+1}) - P_{j+1}}_H^2 +
    2 \sum_{j=0}^{J-1} \Big(
      \int_{t_j}^{t_{j+1}} \ssnm{p(t)}_{H^1} \, \mathrm{d}t
    \Big)^2 + \ssnm{(I - \mathcal P_\tau)z}_{L^2(0,T;H)}^2.
  \end{align*}
  \end{small}
  which, together with the equality
  \[ 
    \sum_{j=0}^{J-1} \ssnmB{
      \int_{t_j}^{t_{j+1}} (z-Z)(t) \, \mathrm{d}W(t)
    }_H^2 = \sum_{j=0}^{J-1}
    \ssnm{z-Z}_{L^2(t_j,t_{j+1};H)}^2 =
    \ssnm{z-Z}_{L^2(0,T;H)}^2,
  \]
  implies
  \begin{equation}
    \label{eq:43}
    \begin{aligned}
      & \ssnm{z-Z}_{L^2(0,T;H)}^2 \\
      \leqslant{} &
      2 \sum_{j=0}^{J-1} \ssnm{p(t_{j+1}) - P_{j+1}}_H^2 +
      2 \sum_{j=0}^{j-1} \Big(
      \int_{t_j}^{t_{j+1}} \ssnm{p(t)}_{H^1} \, \mathrm{d}t
      \Big)^2 + \ssnm{(I - \mathcal P_\tau)z}_{L^2(0,T;H)}^2.
    \end{aligned}
  \end{equation}
  By \cref{eq:xy}, the fact
  \[
    p(t) = \mathbb E_t \Big(
    e^{(T-t)A} p_T + \int_t^T e^{(s-t)A}g(s) \, \mathrm{d}s
    \Big), \quad 0 \leqslant t \leqslant T,
  \]
  and \cref{eq:auxi}, we obtain
  \begin{equation}
    \label{eq:732}
    \sum_{j=0}^{J-1} \ssnm{p(t_{j+1}) - P_{j+1}}_H^2
    \leqslant c \tau \big(
    \ssnm{p_T}_{H^{1/2}}^2 + \ssnm{g}_{L^2(0,T;H)}^2
    \big).
  \end{equation}
  Finally, combining \cref{eq:43,eq:732,eq:Ap+g} proves \cref{eq:lin-conv-z-Z}
  and thus concludes the proof of \cref{lem:lin-conv}.
\hfill\ensuremath{
  \vbox{\hrule height0.6pt\hbox{%
    \vrule height1.3ex width0.6pt\hskip0.8ex
    \vrule width0.6pt}\hrule height0.6pt
  }
}

\subsection{Proof of Theorem 3.1} 
For any $ 0 \leqslant j < J $, since \cref{eq:Itauj-pro-4} implies
\[ 
  \tau^{1/2} \ssnm{\mathcal I_\tau^jv}_H =
  \ssnm{\delta W_j \mathcal I_\tau^j v}_H \leqslant
  \ssnm{v}_{H} \quad \forall v \in L^2(\Omega; H),
\]
we obtain
\[ 
  \nm{\mathcal I_\tau^j}_{
    \mathcal L(L^2(\Omega; H))
  } \leqslant \tau^{-1/2} \quad \forall 0 \leqslant j < J.
\]
By the above estimate, \cref{eq:f-Lips} and the condition $ \tau <
1/C_\text{L}^2 $, a straightforward contraction argument proves that the
temporal semi-discretization \cref{eq:discretization1} admits a unique solution
$ (P,Z) $. In the sequel, we will assume that $ \tau $ is sufficiently small;
otherwise, the error estimate \cref{eq:conv} is evident. We split the rest
of the proof into the following four steps.

{\it Step 1}. We present some preliminary notations and estimates. Let
\begin{equation}
  \label{eq:calM}
  \mathcal M := \tau + \ssnm{(I - \mathcal P_\tau)z}_{L^2(0,T;H)}^2.
\end{equation}
Define $ (\widetilde P, \widetilde Z) \in \mathcal X_\tau \times \mathcal X_\tau
$ by
\begin{subequations}
\label{eq:wtp-wtz}
\begin{numcases}{}
  \widetilde P_J = p_T, \\
  \widetilde Z_j = \mathcal I_\tau^j \Big(
    \widetilde P_{j+1} + \int_{t_j}^{t_{j+1}}
    f(t,p(t),z(t)) \, \mathrm{d}t
  \Big), \quad 0 \leqslant j < J, \\
  \widetilde P_j - \mathbb E_{t_j} \widetilde P_{j+1} =
  \tau A \widetilde P_j +
  \mathbb E_{t_j} \int_{t_j}^{t_{j+1}} f(t,p(t),z(t)) \, \mathrm{d}t,
  \,\, 0 \leqslant j < J.
\end{numcases}
\end{subequations}
In view of $ p_T \in H^{1/2} $ and the fact
\[
  f(\cdot,p(\cdot),z(\cdot)) \in L_\mathbb F^2(0,T;H),
\]
by \cref{lem:lin-conv} we obtain
\begin{equation} 
  \label{eq:p-wtP}
  \begin{aligned}
    & \max_{0 \leqslant k < J} \ssnmb{p(t_k) - \widetilde P_k}_{H} +
    \Big(
      \sum_{k=0}^{J-1} \ssnmb{p - \widetilde P_{k+1}}_{L^2(t_k,t_{k+1};H)}^2
    \Big)^{1/2}  \\
    & \qquad\qquad\qquad {} +
    \ssnmb{z - \widetilde Z}_{L^2(0,T;H)}
    \leqslant c \mathcal M^{1/2}.
  \end{aligned}
\end{equation}
Letting $ E^P := P - \widetilde P $ and $ E^Z := Z - \widetilde
Z $, from \eqref{eq:discretization1} and \eqref{eq:wtp-wtz} we conclude that
\begin{small} 
\begin{subequations}
\label{eq:EP-EZ}
\begin{numcases}{}
  E_J^P = 0, \label{eq:EJP} \\
  E^Z_j = \mathcal I_\tau^j \Big(
    E^P_{j+1} + \int_{t_j}^{t_{j+1}} G(t,E_{j+1}^P,E^Z_j) \, \mathrm{d}t
  \Big), \quad 0 \leqslant j < J, \label{eq:EZ-ZE-j} \\
  E^P_j - \mathbb E_{t_j} E^P_{j+1} =
  \tau A E^P_j +
  \mathbb E_{t_j} \int_{t_j}^{t_{j+1}}  G(t,E^P_{j+1},E^Z_j) \, \mathrm{d}t,
  \, 0 \leqslant j < J, \label{eq:EP-EZ-j}
\end{numcases}
\end{subequations}
\end{small}
where
\begin{equation}
  \label{eq:G-def}
  G(t,E_{j+1}^P, E_j^Z) := f(t, E_{j+1}^P + \widetilde P_{j+1},
  E_j^Z + \widetilde Z_j)
  - f(t, p(t), z(t))
\end{equation}
for all $ t_j \leqslant t < t_{j+1} $ with $ 0 \leqslant j < J $.
We have, for any $ 0 \leqslant k < J $,
\begin{small}
  \begin{align*} 
    & \int_{t_k}^{t_{k+1}}
    \ssnm{G(t,E_{k+1}^P,E_k^Z)}_{H}^2 \, \mathrm{d}t \\
    ={}& \int_{t_k}^{t_{k+1}}
    \ssnmb{
      f(t, E_{k+1}^P + \widetilde P_{k+1},
      E_k^Z + \widetilde Z_k) - f(t, p(t), z(t))
    }_{H}^2 \, \mathrm{d}t \quad \text{(by \cref{eq:G-def})} \\
    \leqslant{} &
    c \int_{t_k}^{t_{k+1}}
    \ssnmb{E_{k+1}^P + \widetilde P_{k+1} - p(t)}_{H}^2 +
    \ssnmb{E_k^Z + \widetilde Z_k - z(t)}_{H}^2 \, \mathrm{d}t
    \quad \text{(by \cref{eq:f-Lips})} \\
    \leqslant{} &
    c \Big(
    \tau \ssnmb{E_{k+1}^P}_H^2 +
    \ssnmb{E^Z}_{L^2(t_k,t_{k+1};H)}^2 +
    \ssnmb{p \!-\! \widetilde P_{k+1}}_{L^2(t_k,t_{k+1};H)}^2 +
    \ssnmb{z \!-\! \widetilde Z}_{L^2(t_k,t_{k+1};H)}^2
    \Big).
  \end{align*}
\end{small}
Hence, for each $ 0 \leqslant j < J $,
\begin{small}
\begin{align*} 
  & \sum_{k=j}^{J-1} \int_{t_k}^{t_{k+1}}
  \ssnmb{G(t, E_{k+1}^P, E_k^Z)}_{H}^2 \, \mathrm{d}t \\
  \leqslant{} & c \Big(
    \sum_{k=j}^{J-1} \tau \ssnmb{E_{k+1}^P}_H^2 +
    \ssnmb{E^Z}_{L^2(t_j,T;H)}^2 + \sum_{k=j}^{J-1}
    \ssnmb{p \!-\! \widetilde P_{k+1}}_{L^2(t_k,t_{k+1};H)}^2 +
    \ssnmb{z \!-\! \widetilde Z}_{L^2(0,T;H)}^2 \!
  \Big),
\end{align*}
\end{small}
which, together with the fact $ E^P_J = 0 $ and \cref{eq:p-wtP}, leads to
\begin{small}
\begin{equation} 
  \label{eq:G-esti}
  \begin{aligned}
    & \sum_{k=j}^{J-1} \int_{t_k}^{t_{k+1}}
    \ssnmb{G(t, E_{k+1}^P, E_k^Z)}_{H}^2 \, \mathrm{d}t \\
    \leqslant{} & c \Big(
    \ssnmb{E^P}_{L^2(t_j,T;H)}^2 +
    \ssnmb{E^Z}_{L^2(t_j,T;H)}^2 +
    \mathcal M
    \Big).
  \end{aligned}
\end{equation}
\end{small}

{\it Step 2}. Let us prove that, for any $ 0 \leqslant j < J $,
\begin{equation}
  \label{eq:E^Z-stab}
  \begin{aligned}
    \ssnmb{E^Z}_{L^2(t_j,T;H)} & \leqslant
    c \Big(
      \mathcal M^{1/2} +
      \ssnm{E^P_J}_{H^{1/2}} +
      \ssnm{E^P}_{L^2(t_j,T;H)} \\
      & \qquad\qquad\qquad\qquad {} +
      \sqrt{T-t_j} \ssnm{E^Z}_{L^2(t_j,T;H)}
    \Big).
  \end{aligned}
\end{equation}
For each $ 0 \leqslant j < J $, define
\begin{equation} 
  \label{eq:etaj-def}
  \eta_j := E^P_{j+1} - E^P_{j} + \tau A E^P_j +
  \int_{t_j}^{t_{j+1}} G(t,E^P_{j+1},E^Z_j) \, \mathrm{d}t.
\end{equation}
Using \cref{eq:Itauj-pro-3}, \eqref{eq:EZ-ZE-j} and the fact
\[ 
  \big[ \tau A E^P_j - E^P_j, \, E^Z_j \delta W_j \big] = 0,
\]
we obtain
\[ 
  \big[
    \eta_j - E^Z_j \delta W_j, \,
    E^Z_j \delta W_j
  \big] = 0
  \quad \text{for all} \quad 0 \leqslant j < J.
\]
For any $ 0 \leqslant k \neq j < J $, since \eqref{eq:EP-EZ-j} implies $ \mathbb
E_{t_j} \eta_j = 0 $ $ \mathbb P $-a.s., it is easily verified that
\[ 
  \big[
    \eta_j - E^Z_j \delta W_j, \,
    E^Z_k \delta W_k
  \big] = 0.
\]
Consequently,
\[ 
  \bigg[
    \sum_{k=j}^{J-1} \eta_k - \sum_{k=j}^{J-1} E^Z_k \delta W_k, \,
    \sum_{k=j}^{J-1} E^Z_k \delta W_k
  \bigg] = 0 \quad \forall 0 \leqslant j < J.
\]
It follows that, for any $ 0 \leqslant j < J $,
\begin{align*} 
  & \ssnmB{
    \sum_{k=j}^{J-1} E^Z_k \delta W_k
  }_{H}^2 =
  \Big[
    \sum_{k=j}^{J-1}
    \eta_k, \, \sum_{k=j}^{J-1} E^Z_k \delta W_k
  \Big] \\
  ={} &
  \Big[
    E^P_J - E^P_j + \sum_{k=j}^{J-1} \tau A E^P_k +
    \sum_{k=j}^{J-1} \int_{t_k}^{t_{k+1}}
    G(t, E^P_{k+1},E^Z_k) \, \mathrm{d}t, \,
    \sum_{k=j}^{J-1} E^Z_k \delta W_k
  \Big] \quad\text{(by \cref{eq:etaj-def})} \\
  ={} &
  \Big[
    E^P_J + \sum_{k=j}^{J-1} \tau A E^P_k +
    \sum_{k=j}^{J-1} \int_{t_k}^{t_{k+1}}
    G(t, E^P_{k+1},E^Z_k) \, \mathrm{d}t, \,
    \sum_{k=j}^{J-1} E^Z_k \delta W_k
  \Big] \\
  \leqslant{} &
  \Big(
    \ssnmb{E^P_J}_{H} \!+\!
    \tau \sum_{k=j}^{J-1} \ssnmb{E^P_k}_{H^1} \!+\!
    \sum_{k=j}^{J-1} \int_{t_k}^{t_{k+1}}
    \ssnmb{G(t,E^P_{k+1},E^Z_k)}_{H} \, \mathrm{d}t
  \Big) \ssnmB{
    \sum_{k=j}^{J-1} E^Z_k \delta W_k
  }_{H},
\end{align*}
which, together with the identity
\begin{align*} 
  \ssnmB{\sum_{k=j}^{J-1} E^Z_k \delta W_k}_H^2 =
  \sum_{k=j}^{J-1} \ssnmb{E^Z_k \delta W_k}_H^2 =
  \sum_{k=j}^{J-1} \tau \ssnmb{E_k^Z}_{H}^2 =
  \ssnmb{E^Z}_{L^2(t_j,T;H)}^2,
\end{align*}
implies
\begin{small}
\begin{align}
  \ssnmb{E^Z}_{L^2(t_j,T;H)}
  & \leqslant \ssnmb{E^P_J}_{H} +
  \tau \sum_{k=j}^{J-1} \ssnmb{E^P_k}_{H^1} +
  \sum_{k=j}^{J-1} \int_{t_k}^{t_{k+1}}
  \ssnmb{G(t,E^P_{k+1},E^Z_k)}_{H} \, \mathrm{d}t \notag \\
  & \leqslant
  \ssnmb{E^P_J}_{H} +
  \sqrt{T-t_j} \ssnmb{E^P}_{L^2(t_j,T;H^1)} + {} \notag \\
  & \qquad \sqrt{T-t_j} \Big(
    \sum_{k=j}^{J-1} \int_{t_k}^{t_{k+1}}
    \ssnmb{G(t,E^P_{k+1},E^Z_k)}_H^2 \, \mathrm{d}t
  \Big)^{1/2}. \label{eq:E^Z-1}
\end{align}
\end{small}
For any $ 0 \leqslant j < J $, it is easily verified by \cref{eq:EP-EZ} that
\begin{align*} 
  E_j^P = \mathbb E_{t_j} \Big(
    (I - \tau A)^{-(J-j)} E^P_J +
    \sum_{k=j}^{J-1} \int_{t_k}^{t_{k+1}}
    (I - \tau A)^{-(k-j+1)} G(t, E^P_{k+1},E^Z_k) \, \mathrm{d}t
  \Big),
\end{align*}
and so using \cref{eq:auxi_stab} gives
\begin{small}
\begin{equation}
  \label{eq:731}
  \begin{aligned}
    & \max_{j \leqslant k < J} \ssnmb{E^P_k}_{H^{1/2}}^2 +
    \ssnmb{E^P}_{L^2(t_j,T;H^1)}^2 \\
    \leqslant{} & c \Big(
      \ssnmb{E^P_J}_{H^{1/2}}^2 +
      \sum_{k=j}^{J-1} \int_{t_k}^{t_{k+1}}
      \ssnmb{G(t,E^P_{k+1}, E^Z_k)}_{H}^2 \, \mathrm{d}t
    \Big).
  \end{aligned}
\end{equation}
\end{small}
Combining \cref{eq:E^Z-1,eq:731} yields, for any $ 0 \leqslant j < J $,
\begin{small}
\[
  \ssnmb{E^Z}_{L^2(t_j,T;H)} \leqslant
  c \ssnmb{E^P_J}_{H^{1/2}} + c \sqrt{T-t_j}
  \Big(
    \sum_{k=j}^{J-1} \int_{t_k}^{t_{k+1}}
    \ssnmb{G(t,E^P_{k+1},E^Z_k)}_{H}^2 \, \mathrm{d}t
  \Big)^{1/2},
\]
\end{small}
so that from \cref{eq:G-esti} we conclude the desired estimate
\cref{eq:E^Z-stab}.

{\it Step 3}. Let $ c^* $ be a particular constant $ c $ in the inequality
\cref{eq:E^Z-stab}, and set
\[ 
  j^* := \min\big\{
    0 \leqslant j < J \mid c^* \sqrt{T-t_j} \leqslant 1/2
  \big\}.
\]
From \cref{eq:E^Z-stab} it follows that
\begin{equation}
  \label{eq:427}
  \ssnmb{E^Z}_{L^2(t_j,T;H)}^2 \leqslant
  c \big(
    \ssnmb{E^P_J}_{H^{1/2}}^2 +
    \ssnmb{E^P}_{L^2(t_j,T;H)}^2 +
    \mathcal M
  \big) \quad \forall j^* \leqslant j < J,
\end{equation}
and so by \cref{eq:G-esti,eq:731} we infer that
\begin{align*}
  \ssnm{E^P_j}_{H^{1/2}}^2 \leqslant
  c \big(
    \ssnmb{E^P_J}_{H^{1/2}}^2 +
    \ssnmb{E^P}_{L^2(t_j,T;H)}^2 + \mathcal M
  \big) \quad \forall j^* \leqslant j < J.
\end{align*}
Since $ H^{1/2} $ is continuously embedded into $ H $, we then obtain
\begin{align*}
  \ssnm{E^P_j}_{H^{1/2}}^2 \leqslant
  c \big(
    \ssnmb{E^P_J}_{H^{1/2}}^2 +
    \ssnmb{E^P}_{L^2(t_j,T;H^{1/2})}^2 + \mathcal M
  \big) \quad \forall j^* \leqslant j < J,
\end{align*}
and therefore using the discrete Gronwall's inequality yields
\[ 
  \max_{j^* \leqslant j \leqslant J}
  \ssnmb{E^P_j}_{H^{1/2}}^2 \leqslant
  c \big(
  \ssnmb{E^P_J}_{H^{1/2}}^2 + \mathcal M
  \big),
\]
which, together with \cref{eq:427}, leads to
\[ 
  \max_{j^* \leqslant j < J} \ssnmb{E^P_j}_{H^{1/2}} +
  \ssnmb{E^Z}_{L^2(t_{j^*}, T; H)}
  \leqslant c \ssnmb{E^P_J}_{H^{1/2}} +
  c \mathcal M^{1/2}.
\]
Hence, by the estimate $ \ssnmb{E_J^P}_{H^{1/2}} \leqslant c \mathcal M^{1/2} $
(in fact $ E_J^P = 0 $), we obtain
\begin{equation} 
  \max_{j^* \leqslant j < J} \ssnmb{E^P_j}_{H^{1/2}} +
  \ssnmb{E^Z}_{L^2(t_{j^*}, T; H)}
  \leqslant c \mathcal M^{1/2}.
\end{equation}

{\it Step 4}. Note that $ J/(J-j^*) $ is independent of $ \tau $. Repeating the
argument in Steps 2 and 3 several times (not greater than $ J/(J-j^*) $)
proves
\begin{equation} 
  \label{eq:conv-auxi-final}
  \max_{0 \leqslant j < J} \ssnmb{E^P_j}_{H^{1/2}} +
  \ssnmb{E^Z}_{L^2(0, T; H)}
  \leqslant c \mathcal M^{1/2},
\end{equation}
which, together with \cref{eq:p-wtP} and the fact that $ H^{1/2} $ is
continuously embedded into $ H $, yields the desired estimate \cref{eq:conv}.
This completes the proof of \cref{thm:conv1}.

\begin{remark}
  Assume that $ (P,Z) $ is the solution to \cref{eq:discretization1} and that $
  f $ satisfies $ (i) $ and $ (ii) $ in \cref{hypo:main}. Using the techniques
  in the proof of \cref{thm:conv1}, we can easily obtain the following stability
  estimate:
  \[ 
    \max_{0 \leqslant j \leqslant J}
    \ssnm{P_j}_{H^{1/2}} + \ssnm{Z}_{L^2(0,T;H)}
    \leqslant c \big(
      \ssnm{p_T}_{H^{1/2}} + \ssnm{f(\cdot,0,0)}_{L^2(0,T;H)}
    \big),
  \]
  provided that $ p_T \in H^{1/2} $. Moreover, we can use the estimate
  \cref{eq:conv-auxi-final}
  and the stability estimate of $ \widetilde P $ to further derive the stability
  estimate of $ P $ for $ p_T \in H $.
\end{remark}

\begin{remark}
  \label{rem:tmp}
  Following the proof of \cref{thm:conv1}, we can easily prove
  \cref{thm:conv2,thm:conv3} by \cref{lem:lin-conv2,lem:lin-conv3},
  respectively.
\end{remark}

\section{Application to a stochastic linear quadratic control problem}
\label{sec:application}
\subsection{Continuous problem}
We are concerned with the following stochastic linear quadratic control problem:
\begin{equation}
  \label{eq:optim}
  \min_{
    u \in L_\mathbb F^2(0,T;H)
  } \frac12 \ssnm{y-y_d}_{L^2(0,T; H)}^2 +
  \frac{\nu}2 \ssnm{u}_{L^2(0,T; H)}^2,
\end{equation}
subject to the state equation
\begin{equation}
  \label{eq:state}
  \begin{cases}
    \mathrm{d}y(t) = (A y + \alpha_0 y + \alpha_1 u)(t) \, \mathrm{d}t +
    (\alpha_2 y + \alpha_3 u)(t) \, \mathrm{d}W(t),
    & 0 \leqslant t \leqslant T, \\
    y(0) = 0,
  \end{cases}
\end{equation}
where $ 0 < \nu < \infty $, $ y_d \in L_\mathbb F^2(0,T;H) $ and
\[ 
  \alpha_0, \alpha_1, \alpha_2, \alpha_3 \in
  L_\mathbb F^2(0,T;\mathbb R) \cap
  L^\infty(\Omega \times (0,T)).
\]
It is standard that problem \cref{eq:optim} admits a
unique solution $ \bar u $. Let $ \bar y $ be the state with respect to the
control $ \bar u $, and let $ (\bar p, \bar z) $ be the solution of the backward
stochastic evolution equation
\begin{equation} 
  \label{eq:barp-barz}
  \begin{cases}
    \mathrm{d}\bar p(t) = -(
    A \bar p + \alpha_0 \bar p +
    \bar y - y_d + \alpha_2 \bar z
    )(t) \, \mathrm{d}t +
    \bar z(t) \, \mathrm{d}W(t), \quad 0 \leqslant t \leqslant T, \\
    \bar p(T) = 0.
  \end{cases}
\end{equation}
Applying the celebrated It\^o's formula to $ [y(\cdot), \bar p(\cdot)] $ yields
\[ 
  \int_0^T \big[ (\bar y - y_d)(t), y(t) \big]
  \, \mathrm{d}t =
  \int_0^T \big[
    (\alpha_1 \bar p + \alpha_3 \bar z)(t), u(t)
  \big] \, \mathrm{d}t
\]
for all $ u \in L_\mathbb F^2(0,T;H) $, where $ y $ is the state with respect to
the control $ u $. Using the above equality, we readily conclude the first-order
optimality condition of problem \cref{eq:optim}:
\begin{equation} 
  \label{eq:optim_cond}
  \bar u = - \nu^{-1}(\alpha_1 \bar p + \alpha_3 \bar z).
\end{equation}
Noting that $ (\bar p, \bar z) $ is the solution to \cref{eq:barp-barz}, we have
\begin{equation}
  \label{eq:barp-barz-regu}
  (\bar p, \bar z) \in \big(
  L_\mathbb F^2(\Omega;C([0,T];H^{1/2})) \cap L_\mathbb F^2(0,T;H^1)
  \big) \times L_\mathbb F^2(0,T;H^{1/2}),
\end{equation}
and so by \cref{eq:optim_cond} we get
\[
  \bar u \in L_\mathbb F^2(0,T;H^{1/2}).
\]
Since $ \bar y $ is the state with respect to the control $ \bar u $, we then
obtain
\begin{equation} 
  \label{eq:optim_regu}
  \bar y \in L_\mathbb F^2(\Omega;C([0,T];H^{1/2}))
  \cap L_\mathbb F^2(0,T;H^1).
\end{equation}

\begin{remark} 
  The first-order optimality condition \cref{eq:optim_cond} follows from
  \cite{Bensoussan1983,Bismut1973}. For the theoretical analysis of the
  stochastic linear quadratic control problems in infinite dimensions, we refer
  the reader to \cite{LuZhangbook2021} and the references therein.
\end{remark}

\begin{remark} 
  The regularity results \cref{eq:barp-barz-regu,eq:optim_regu} are
  straightforward by the Galerkin method and the standard theory of the
  stochastic differential equations and the backward stochastic differential
  equations (see~\cite[Chapters 3 and 5]{Pardoux2014}).
\end{remark}

\subsection{Temporally semi-discrete problem}
The temporally semi-discrete problem reads as follows:
\begin{equation}
  \label{eq:optim_h}
  \min_{
    U \in \mathcal X_\tau
  } \frac12 \ssnm{Y-y_d}_{L^2(0,T; H)}^2 +
  \frac{\nu}2 \ssnm{U}_{L^2(0,T; H)}^2,
\end{equation}
subject to the discrete state equation
\begin{equation} 
  \label{eq:state_h}
  \begin{cases}
    Y_{j+1} -  Y_j = \tau AY_{j+1} +
    \int_{t_j}^{t_{j+1}} (\alpha_0  Y + \alpha_1 U)(t)
    \, \mathrm{d}t + {} \\
    \qquad\qquad\qquad\qquad\qquad\qquad
    \int_{t_j}^{t_{j+1}}
    (\alpha_2  Y + \alpha_3 U)(t) \, \mathrm{d}W(t),
    \qquad 0 \leqslant j < J, \\
    Y_0 = 0,
  \end{cases}
\end{equation}
where $ Y \in \mathcal X_\tau $. The main result of this section is the
following error estimate.
\begin{theorem} 
  \label{thm:optim-conv}
  Assume that $ y_d \in L_\mathbb F^2(0,T;H) $. Let $ \bar u $ and $ \bar U $
  be the solutions to problems \cref{eq:optim,eq:optim_h}, respectively. Then
  \begin{equation} 
    \label{eq:optim-conv}
    \ssnm{\bar u - \bar U}_{L^2(0,T;H)}
    \leqslant c \big( \tau^{1/2} + \nm{(I - \mathcal P_\tau)\bar u}_{L^2(0,T;H)}
      \big).
  \end{equation}
\end{theorem}

\begin{remark} 
  Recently, Li and Xie \cite{LiXie2021} have analyzed a spatial semi-discretization
  for a stochastic linear quadratic control problem with general filtration.
  For a special case of problem \cref{eq:optim_h}, Li and Zhou \cite{LiZhou2020}
  obtained the temporal accuracy $ O(\tau^{1/2}) $ for rough data. For other
  related works, we refer the reader to \cite{Dunst2016,Wang2020a,Wang2020b,ZhouLi2021}.
\end{remark}

The main task of the rest of this subsection is to prove the above theorem. To
this end, we proceed as follows. For any $ v \in L_\mathbb F^2(0,T;H) $, we use
$ S_\tau v $ to denote the solution to discretization \cref{eq:state_h} with $ U
$ being replaced by $ v $. A routine argument (see, e.g.,
\cite[Theorem~3.14]{Krusebook2014}) gives
\begin{equation} 
  \label{eq:S0-stab}
  \max_{0 \leqslant j \leqslant J}
  \ssnm{(S_\tau v)_j}_{H}
  \leqslant c \nm{v}_{L^2(0,T;H)}.
\end{equation}
For any $ P,Z \in \mathcal X_\tau $ and $ g, v \in L_\mathbb F^2(0,T;H) $,
define
\begin{small}
\begin{equation} 
  \label{eq:scrS}
  \begin{aligned}
    \mathscr S(P,Z,g,v)& :=
    \sum_{j=0}^{J-1} \bigg(
      \int_{t_j}^{t_{j+1}}
      \big[
        (\alpha_1 P_{j+1} + \alpha_3 Z)(t), \, v(t)
      \big] \, \mathrm{d}t  {} \\
     & -\Big[
        \int_{t_j}^{t_{j+1}} (\alpha_0 P_{j+1} +
        g + \alpha_2 Z)(t) \, \mathrm{d}t, \,
        \int_{t_j}^{t_{j+1}} (\alpha_2 S_\tau v + \alpha_3 v)(t)  \, \mathrm{d}W(t)
      \Big]
    \bigg).
  \end{aligned}
\end{equation}
\end{small}
In the sequel we will always assume
\[
  \tau < \frac1{
    \nm{\alpha_2}_{L^\infty(\Omega \times (0,T))}^2
  },
\]
to ensure that the later discretizations \cref{eq:barph-barzh,eq:P-Z} each
admit a unique solution (see the proof of \cref{thm:conv1}). One form of the
first-order optimality condition of problem \cref{eq:optim_h} is as follows.
\begin{lemma} 
  Assume that $ \bar U $ is the solution to problem \cref{eq:optim_h}. Let $
  (\bar P, \bar Z) \in \mathcal X_\tau \times L_\mathbb F^2(0,T;H) $ be
  the solution to the discretization
  \begin{subequations} 
  \label{eq:barph-barzh}
  \begin{numcases}{}
    \bar P_J = 0, \\
    \bar P_j - \bar P_{j+1} = \tau A \bar P_j +
    \int_{t_j}^{t_{j+1}} \big(
      \alpha_0 \bar P_{j+1} +
      S_\tau \bar U - y_d + \alpha_2 \bar Z
    \big)(t) \, \mathrm{d}t \notag \\
    \qquad\qquad\qquad\qquad\qquad\qquad {} -
    \int_{t_j}^{t_{j+1}} \bar Z(t) \, \mathrm{d}W(t),
    \quad 0 \leqslant j < J.
  \end{numcases}
  \end{subequations}
  Then
  \begin{equation} 
    \label{eq:optim_cond_h}
    \nu \int_0^T [\bar U(t), U(t)] \, \mathrm{d}t +
    \mathscr S(\bar P, \bar Z, S_\tau \bar U - y_d, U) = 0
    \quad \forall U \in \mathcal X_\tau.
  \end{equation}
\end{lemma}
\begin{proof} 
  Following the proof of \cite[Lemma 4.19] {LiZhou2020}, we can easily obtain
  \begin{equation} 
    \label{eq:foo}
    \int_0^T \big[ (S_\tau  \bar U - y_d)(t), (S_\tau v)(t) \big]
    \, \mathrm{d}t = \mathscr S(\bar P, \bar Z, S_\tau  \bar U - y_d, v)
  \end{equation}
  for all $ v \in L_\mathbb F^2(0,T;H) $. By this equality, a straightforward
  calculation yields \cref{eq:optim_cond_h}.
\end{proof}
\begin{remark} 
  Note that \cref{eq:barph-barzh} is not a natural adjoint equation of the
  discrete state equation \cref{eq:state_h}, and hence the first-order
  optimality condition \cref{eq:optim_cond_h} is unusual. We can also use the
  temporal semi-discretizations \cref{eq:discretization1,eq:discretization2} to
  form the first-order optimality condition of problem \cref{eq:optim_cond_h};
  however, we observe that the temporal semi-discretization
  \cref{eq:discretization3} appears to be more suitable for the numerical
  analysis of problem \cref{eq:optim_h}.
\end{remark}

\begin{lemma}
  Let $ \bar u $ be the solution to \cref{eq:optim}, and let $ \bar y $ be the state
  with respect to $ \bar u $. Then
  \begin{equation}
    \label{eq:I1-esti}
    \ssnm{\bar y - S_\tau \bar u}_{L^2(0,T;H)}
    \leqslant c\tau^{1/2}.
  \end{equation}
\end{lemma}
\begin{proof}
  Fix $ 0 \leqslant j < J $. By definition we have
  \[
    \mathrm{d}\bar y(t) = (A \bar y + \alpha_0 \bar y + \alpha_1 \bar u)(t)
    \, \mathrm{d}t + (\alpha_2 \bar y + \alpha_3 \bar u)(t) \, \mathrm{d}W(t),
    \quad 0 \leqslant t \leqslant T,
  \]
  so that
  \[
    \bar y(t) - \bar y(t_j) = \int_{t_j}^t (A \bar y + \alpha_0\bar y + \alpha_1 \bar u)(t)
    \, \mathrm{d}t + \int_{t_j}^t (\alpha_2 \bar y + \alpha_3 \bar u)(t) \,
    \mathrm{d}W(t), \quad t_j \leqslant t \leqslant T.
  \]
  It follows that for any $ t_j \leqslant t \leqslant t_{j+1} $,
  \begin{align*}
    & \ssnmb{\bar y(t) - \bar y(t_j)}_H^2 \\
    \leqslant{} &
    2 \ssnmB{
      \int_{t_j}^t (A\bar y + \alpha_0 \bar y + \alpha_1 \bar u)(t) \, \mathrm{d}t
    }_H^2 + 2 \ssnmB{
      \int_{t_j}^t (\alpha_2 \bar y + \alpha_3 \bar u)(t) \, \mathrm{d}W(t)
    }_H^2 \\
    ={} &
    2 \ssnmB{
      \int_{t_j}^t (A\bar y + \alpha_0 \bar y + \alpha_1 \bar u)(t) \, \mathrm{d}t
    }_H^2 + 2 \int_{t_j}^t
    \ssnm{(\alpha_2 \bar y + \alpha_3 \bar u)(t)}_H^2 \, \mathrm{d}t \\
    \leqslant{} &
    2(t-t_j) \int_{t_j}^t \ssnm{
      (A\bar y + \alpha_0 \bar y + \alpha_1 \bar u)(t)
    }_H^2 \, \mathrm{d}t + 2 \int_{t_j}^t
    \ssnm{(\alpha_2 \bar y + \alpha_3 \bar u)(t)}_H^2 \, \mathrm{d}t,
  \end{align*}
  which implies
  \begin{align*}
    & \ssnm{\bar y - \bar y(t_j)}_{L^2(t_j,t_{j+1};H)}^2 \\
    \leqslant{} &
    \tau^2 \ssnm{A\bar y + \alpha_0\bar y + \alpha_1\bar u}_{L^2(t_j,t_{j+1};H)} ^2 +
    2\tau \ssnm{\alpha_2 \bar y + \alpha_3 \bar u}_{L^2(t_j,t_{j+1};H)}^2.
  \end{align*}
  Hence,
  \begin{align*}
    & \sum_{j=0}^{J-1} \ssnm{\bar y - \bar y(t_j)}_{L^2(t_j,t_{j+1};H)}^2 \\
    \leqslant{} &
    \tau^2 \ssnm{A\bar y + \alpha_0\bar y + \alpha_1\bar u}_{L^2(0,T;H)} ^2 +
    2\tau \ssnm{\alpha_2 \bar y + \alpha_3 \bar u}_{L^2(0,T;H)}^2.
  \end{align*}
  By \cref{eq:optim_regu} and the fact $ \bar u \in L_\mathbb F^2(0,T;H) $, we
  then obtain
  \[ 
    \sum_{j=0}^{J-1} \ssnm{\bar y - \bar y(t_j)}_{L^2(t_j,t_{j+1};H)}^2
    \leqslant c \tau,
  \]
  so that the desired estimate \cref{eq:I1-esti} follows from
  \[ 
    \max_{0 \leqslant j < J} \ssnm{\bar y(t_j) - (S_\tau \bar u)_j}_{H}
    \leqslant c \tau^{1/2} \quad
    \text{(see \cite[Theorem~3.14]{Krusebook2014}).}
  \]
  This completes the proof.
\end{proof}

Finally, we are in a position to prove \cref{thm:optim-conv} as follows.

\medskip\noindent{\bf Proof of \cref{thm:optim-conv}}. Let $ \bar y $ be the
state with respect to the control $ \bar u $, and let $ (\bar p, \bar z) $ be
the solution to equation \cref{eq:barp-barz}. Similar to \cref{eq:lin2-p-pj}, we
have
\begin{equation} 
  \label{eq:barp-barpj}
  \Big(
    \sum_{j=0}^{J-1} \ssnm{\bar p - \bar p(t_{j+1})}_{
      L^2(t_j,t_{j+1}; H)
    }^2
  \Big)^{1/2} \leqslant c \tau^{1/2}.
\end{equation}
We divide the rest of the proof into the following four steps.

{\it Step 1}. Let $ (P, Z) \in \mathcal X_\tau \times L_\mathbb F^2(0,T; H) $
be the solution to the discretization
\begin{equation}
  \label{eq:P-Z}
  \begin{cases}
    P_J = 0, \\
    P_j -  P_{j+1} = \tau A  P_j +
    \int_{t_j}^{t_{j+1}} \big(
      \alpha_0  P_{j+1} + \bar y - y_d + \alpha_2 Z
    \big)(t) \, \mathrm{d}t \\
    \qquad\qquad\qquad\qquad\qquad\qquad\qquad{} -
    \int_{t_j}^{t_{j+1}}  Z(t) \, \mathrm{d}W(t),
    \quad 0 \leqslant j < J.
  \end{cases}
\end{equation}
In view of \cref{eq:optim_regu} and the fact $ y_d \in L_\mathbb F^2(0,T;
H) $, we can use \cref{thm:conv3} to conclude that
\begin{equation} 
  \label{eq:P-Z-conv}
  \max_{0 \leqslant j \leqslant J}
  \ssnm{\bar p(t_j) - P_j}_{H} +
  \ssnm{\bar z - Z}_{L^2(0,T;H)}
  \leqslant c \tau^{1/2},
\end{equation}
which, together with \cref{eq:barp-barpj}, yields
\begin{small}
  \begin{align}  \label{eq:barp-P}
    & \Big(
      \sum_{j=0}^{J-1} \ssnm{\bar p - P_{j+1}}_{
        L^2(t_j,t_{j+1}; H)
      }^2
    \Big)^{1/2} \notag \\
    \leqslant{} &
    \Big(
      \sum_{j=0}^{J-1} \ssnm{\bar p - \bar p(t_{j+1})}_{
        L^2(t_j,t_{j+1}; H)
      }^2
    \Big)^{1/2} + \Big(
      \sum_{j=0}^{J-1} \ssnm{P_{j+1} - \bar p(t_{j+1})}_{
        L^2(t_j,t_{j+1}; H)
      }^2
    \Big)^{1/2} \notag \\
      \leqslant{} &
      c \tau^{1/2}.
    \end{align}
    \end{small}
    In addition, from \cref{eq:P-Z-conv} and \cref{eq:barp-barz-regu} we conclude that
    \begin{equation} 
      \label{eq:P-Z-stab}
      \Big(
        \tau \sum_{j=0}^{J-1} \ssnm{P_{j+1}}_{H}^2
      \Big)^{1/2} + \ssnm{Z}_{L^2(0,T;H)}
      \leqslant c.
    \end{equation}

    {\it Step 2}. Let us prove
    \begin{equation}
      \label{eq:u-U-I1-I2}
      \nu \ssnm{\bar u - \bar U}_{L^2(0,T;H)}^2
      \leqslant c\tau + c\ssnm{(I - \mathcal P_\tau)\bar u}_{ L^2(0,T;H) }^2
      + \mathbb I_1 + \mathbb I_2 + \mathbb I_3 + \mathbb I_4,
    \end{equation}
    where
    \begin{small}
      \begin{align*} 
        \mathbb I_1
        & := \sum_{j=0}^{J-1} \int_{t_j}^{t_{j+1}}
        \big[
          \alpha_1(t)(P_{j+1} - \bar p(t)), (\mathcal P_\tau\bar u - \bar U)(t)
        \big] \, \mathrm{d}t, \\
        \mathbb I_2
        &:= \int_0^T \big[
          (\alpha_3 Z -  \alpha_3 \bar z)(t),
          (\mathcal P_\tau\bar u - \bar U)(t)
        \big] \,
        \mathrm{d}t, \\
        \mathbb I_3
        &:= -\sum_{j=0}^{J-1} \Big[
          \int_{t_j}^{t_{j+1}} \! \big(
            \alpha_0 P_{j+1} \!+\! \bar y \!-\! y_d \!+\! \alpha_2 Z
          \big)(t) \mathrm{d}t,
          \int_{t_j}^{t_{j+1}} \! \big(
            \alpha_2 S_\tau (\mathcal P_\tau\bar u \!-\! \bar U)
            \!+\! \alpha_3(\mathcal P_\tau\bar u \!-\! \bar U)
          \big)(t) \mathrm{d}W(t)
        \Big], \\
        \mathbb I_4 &:=
        -\int_0^T \Big[ (\alpha_1\bar p + \alpha_3 \bar z)(t),
          (\bar u-\mathcal P_\tau\bar u)(t)
        \Big]
        \, \mathrm{d}t.
      \end{align*}
    \end{small}
    \hskip -4pt The basic idea is standard (see, e.g.,
    \cite[Theore~3.4]{Hinze2009}). We first present three equalities. Inserting $
    v := \mathcal P_\tau\bar u - \bar U $ into \cref{eq:foo} gives
    \begin{equation} 
      \label{eq:230}
      \int_0^T \big[
        (S_\tau \bar U - y_d)(t), \,
        (S_\tau (\mathcal P_\tau\bar u - \bar U))(t)
      \big] \, \mathrm{d}t =
      \mathscr S(\bar P, \bar Z, S_\tau \bar U - y_d, \mathcal P_\tau\bar u - \bar U),
    \end{equation}
    and similarly we have
    \begin{equation} 
      \label{eq:231}
      \int_0^T \big[
        (\bar y-y_d)(t), \,
        (S_\tau (\mathcal P_\tau\bar u-\bar U))(t)
      \big] \, \mathrm{d}t =
      \mathscr S(P, Z, \bar y - y_d, \mathcal P_\tau\bar u - \bar U).
    \end{equation}
    By definition, it is easily verified that
    \begin{equation} 
      \label{eq:232}
      \begin{aligned}
      & \mathscr S(P, Z, \bar y - y_d, \mathcal P_\tau\bar u - \bar U) -
      \int_0^T \big[
        (\alpha_1 \bar p + \alpha_3 \bar z)(t), \,
        (\bar u - \bar U)(t)
      \big] \, \mathrm{d}t \\
        ={} &
        \mathbb I_1 + \mathbb I_2 + \mathbb I_3 + \mathbb I_4.
      \end{aligned}
    \end{equation}
    Next, by \cref{eq:optim_cond} we have
    \[ 
      \nu \int_0^T [\bar u(t), (\bar u - \bar U)(t)] \, \mathrm{d}t =
      - \int_0^T \big[
        (\alpha_1\bar p + \alpha_3 \bar z)(t),
        (\bar u - \bar U)(t)
      \big] \, \mathrm{d}t,
    \]
    and inserting $ U := \mathcal P_\tau\bar u - \bar U $ into
    \cref{eq:optim_cond_h} gives
    \begin{equation} 
      -\nu \int_0^T \big[ \bar U(t), \, (\bar u - \bar U)(t) \big] \, \mathrm{d}t
      = \mathscr S(\bar P, \bar Z, S_\tau \bar U - y_d, \mathcal P_\tau\bar u - \bar U).
    \end{equation}
    Summing up the above two equalities yields
    \begin{align*} 
    & \nu \ssnm{\bar u - \bar U}_{L^2(0,T;H)}^2 \\
      ={} &
      - \int_0^T \big[
        (\alpha_1 \bar p + \alpha_3 \bar z)(t), \,
        (\bar u - \bar U)(t)
      \big] \, \mathrm{d}t +
      \mathscr S(\bar P, \bar Z, S_\tau \bar U - y_d, \mathcal P_\tau\bar u - \bar U) \\
      ={} &
      \mathscr S(P, Z, \bar y - y_d, \mathcal P_\tau\bar u - \bar U) -
      \int_0^T \big[
        (\alpha_1 \bar p + \alpha_3 \bar z)(t), \,
        (\bar u - \bar U)(t)
      \big] \, \mathrm{d}t  {} \\
          & \quad +\mathscr S(\bar P, \bar Z, S_\tau \bar U - y_d, \mathcal P_\tau\bar u - \bar U) -
          \mathscr S(P,Z,\bar y - y_d, \mathcal P_\tau\bar u - \bar U) \\
      ={} &
      \mathbb I_1 + \mathbb I_2 + \mathbb I_3 + \mathbb I_4 +
      \mathscr S(\bar P, \bar Z, S_\tau \bar U - y_d, \mathcal P_\tau\bar u - \bar U) -
      \mathscr S(P,Z,\bar y - y_d, \mathcal P_\tau\bar u - \bar U)
      \quad \text{(by \cref{eq:232})} \\
      ={} &
      \mathbb I_1 + \mathbb I_2 + \mathbb I_3 + \mathbb I_4 +
      \int_0^T \big[
        (S_\tau \bar U - \bar y)(t), (S_\tau(\mathcal P_\tau \bar u - \bar U))(t)
      \big] \, \mathrm{d}t \qquad\text{(by \cref{eq:230,eq:231})}.
    \end{align*}
    Hence, the desired estimate \cref{eq:u-U-I1-I2} follows from
    \begin{align*} 
    & \int_0^T \big[
      (S_\tau \bar U - \bar y)(t), \,
      (S_\tau (\mathcal P_\tau\bar u - \bar U))(t)
    \big] \, \mathrm{d}t \\
      ={} &
      -\ssnm{\bar y - S_\tau \bar U}_{L^2(0,T;H)}^2 +
      \int_0^T \big[
        (S_\tau \bar U - \bar y)(t),
        (S_\tau\mathcal P_\tau\bar u - \bar y)(t)
      \big] \, \mathrm{d}t \\
      \leqslant{} &
      -\frac12 \ssnm{\bar y - S_\tau \bar U}_{L^2(0,T;H)}^2 +
      \frac12 \ssnm{\bar y - S_\tau\mathcal P_\tau\bar u}_{L^2(0,T;H)}^2 \\
      \leqslant{} &
      \frac12 \ssnm{\bar y - S_\tau\mathcal P_\tau\bar u}_{L^2(0,T;H)}^2 \\
      \leqslant{} &
      \ssnm{\bar y - S_\tau\bar u}_{L^2(0,T;H)}^2 +
      \ssnm{S_\tau(I-\mathcal P_\tau)\bar u}_{L^2(0,T;H)}^2 \\
      \leqslant{} &
      c \tau + c \ssnm{(I-\mathcal P_\tau)\bar u}_{L^2(0,T;H)}^2
      \quad \text{(by \cref{eq:I1-esti,eq:S0-stab})}.
    \end{align*}

    {\it Step 3}. Let us estimate $ \mathbb I_1 $, $ \mathbb I_2 $, $ \mathbb I_3
    $ and $ \mathbb I_4 $. For $ \mathbb I_1, $ by \cref{eq:barp-P} we have
    \begin{align*}
      \mathbb I_1
    & \leqslant
    c \Big(
      \sum_{j=0}^{J-1} \ssnm{\bar p - P_{j+1}}_{
        L^2(t_j,t_{j+1};H)
      }^2
    \Big)^{1/2} \ssnm{\mathcal P_\tau\bar u - \bar U}_{L^2(0,T;H)} \\
      &\leqslant c \tau^{1/2} \ssnm{\mathcal P_\tau\bar u - \bar U}_{L^2(0,T;H)}.
  \end{align*}
  For $ \mathbb I_2 $ we have
  \begin{align*}
    \mathbb I_2
    &\leqslant
    c \ssnm{\bar z - Z}_{L^2(0,T;H)}
    \ssnm{\mathcal P_\tau\bar u - \bar U}_{L^2(0,T;H)} \\
    &\leqslant c \tau^{1/2} \ssnm{\mathcal P_\tau\bar u - \bar U}_{L^2(0,T;H)}
    \quad\text{(by \cref{eq:P-Z-conv})}.
  \end{align*}
  For $ \mathbb I_3 $ we have
  \begin{small}
    \begin{align*} 
      \mathbb I_3
      & \leqslant
      \sum_{j=0}^{J-1} \! \ssnmB{\!
      \int_{t_j}^{t_{j+1}} \!\! (\alpha_0 P_{j+1} \!+\!
      \bar y \!-\! y_d \!+\! \alpha_2 Z)(t) \mathrm{d}t
      }_{H} \ssnmB{\!
      \int_{t_j}^{t_{j+1}} \!\! \big(\!
      \alpha_2S_\tau (\mathcal P_\tau\bar u \!-\! \bar U) +
      \alpha_3 (\mathcal P_\tau\bar u \!-\! \bar U) \!
      \big)(t)  \mathrm{d}W(t)
      }_{H} \\
      & = \sum_{j=0}^{J-1} \ssnmB{
        \int_{t_j}^{t_{j+1}} (\alpha_0 P_{j+1} \!+\!
        \bar y \!-\! y_d \!+\! \alpha_2 Z)(t) \mathrm{d}t
      }_{H} \ssnm{
        \alpha_2S_\tau (\mathcal P_\tau\bar u - \bar U) \!+\! \alpha_3 (\mathcal
        P_\tau\bar u - \bar U)
      }_{L^2(t_j,t_{j+1};H)} \\
      & \leqslant \Big(
      \sum_{j=0}^{J-1}
      \ssnmB{
        \int_{t_j}^{t_{j+1}} \! (\alpha_0 P_{j+1} \!+\!
        \bar y \!-\! y_d \!+\! \alpha_2 Z)(t) \mathrm{d}t
      }_{H}^2
      \Big)^{1/2} \ssnm{
        \alpha_0 S_\tau (\mathcal P_\tau\bar u \!-\! \bar U) +
        \alpha_3(\mathcal P_\tau\bar u \!-\! \bar U)
      }_{L^2(0,T;H)} \\
      & \leqslant c\Big(
      \sum_{j=0}^{J-1}
      \ssnmB{
        \int_{t_j}^{t_{j+1}} (\alpha_0 P_{j+1} +
        \bar y - y_d + \alpha_2 Z)(t) \mathrm{d}t
      }_{H}^2
      \Big)^{1/2} \ssnm{\mathcal P_\tau\bar u - \bar U }_{L^2(0,T;H)}
      \quad \text{(by \cref{eq:S0-stab})} \\
      & \leqslant c \sqrt\tau \ssnm{\mathcal P_\tau\bar u - \bar U}_{L^2(0,T;H)},
    \end{align*}
  \end{small}
  since 
  \begin{small}
    \begin{align*} 
      & \Big(
      \sum_{j=0}^{J-1}
      \ssnmB{
        \int_{t_j}^{t_{j+1}} \big(
        \alpha_0 P_{j+1} + \bar y - y_d + \alpha_2 Z
        \big)(t) \, \mathrm{d}t
      }_{H}^2
      \Big)^{1/2} \\
      \leqslant{} &
      \Big(
      \sum_{j=0}^{J-1}
      \Big(
      \int_{t_j}^{t_{j+1}} \ssnm{
        (\alpha_0 P_{j+1} + \bar y - y_d + \alpha_2 Z)(t)
      }_{H} \, \mathrm{d}t
      \Big)^2
      \Big)^{1/2} \\
      \leqslant{} &
      \Big(
      \sum_{j=0}^{J-1}
      \tau \int_{t_j}^{t_{j+1}} \ssnm{
        (\alpha_0 P_{j+1} + \bar y - y_d + \alpha_2 Z)(t)
      }_{H}^2 \, \mathrm{d}t
      \Big)^{1/2} \\
      ={} &
      \sqrt\tau \Big(
      \sum_{j=0}^{J-1}
      \ssnm{\alpha_0 P_{j+1} + \bar y - y_d + \alpha_2 Z}_{
        L^2(t_j,t_{j+1};H)
      }^2
      \Big)^{1/2} \\
      \leqslant{} &
      c \sqrt\tau \quad \text{(by \cref{eq:P-Z-stab})}.
    \end{align*}
  \end{small}
  For $ \mathbb I_4 $, by \cref{eq:optim_cond} and the definition of $ \mathcal
  P_\tau $ we have
  \[ 
    \mathbb I_4 = \nu\int_0^T \big[
      \bar u(t), (\bar u - \mathcal P_\tau \bar u)(t)
      \big] \, \mathrm{d}t = \nu
    \ssnm{(I - \mathcal P_\tau)\bar u}_{L^2(0,T;H)}^2.
  \]

{\it Step 4}. Combining \cref{eq:u-U-I1-I2} and the above estimates of $ \mathbb
I_1 $, $ \mathbb I_2 $, $ \mathbb I_3 $ and $ \mathbb I_4 $ in Step 3, we conclude that
\begin{small}
\begin{align*}
  & \nu\ssnm{\bar u - \bar U}_{L^2(0,T;H)}^2 \\
  \leqslant{} &
  c\tau + c \ssnm{(I-\mathcal P_\tau)\bar u}_{L^2(0,T;H)}^2 +
  c \tau^{1/2} \ssnm{\mathcal P_\tau\bar u - \bar U}_{L^2(0,T;H)} \\
  \leqslant{} &
  c\tau + c \ssnm{(I-\mathcal P_\tau)\bar u}_{L^2(0,T;H)}^2 +
  c \tau^{1/2} \ssnm{(I - \mathcal P_\tau)\bar u}_{L^2(0,T;H)} +
  c\tau^{1/2}\ssnm{\bar u - \bar U}_{L^2(0,T;H)} \\
  \leqslant{} &
  c\tau + c \ssnm{(I-\mathcal P_\tau)\bar u}_{L^2(0,T;H)}^2 +
  c\tau^{1/2}\ssnm{\bar u - \bar U}_{L^2(0,T;H)}.
\end{align*}
\end{small}
We can then apply the Young's inequality with $ \varepsilon $ to obtain
\[
  \nu \ssnm{\bar u - \bar U}_{L^2(0,T;H)}^2
  \leqslant c \tau + c\ssnm{(I - \mathcal P_\tau)\bar u}_{L^2(0,T;H)}^2,
\]
which implies the desired estimate \cref{eq:optim-conv}. This completes the
proof of \cref{thm:optim-conv}.

\hfill\ensuremath{
  \vbox{\hrule height0.6pt\hbox{%
    \vrule height1.3ex width0.6pt\hskip0.8ex
    \vrule width0.6pt}\hrule height0.6pt
  }
}

\section{Conclusions}
\label{sec:conclusion} 
In this paper, we have analyzed three Euler type temporal semi-discretizations
for a backward semilinear stochastic evolution equation with Lipschitz
nonlinearity. With reasonable regularity assumptions on the data, we have
established the convergence for the first two semi-discretizations and derived
an explicit convergence rate for the third semi-discretization. In the numerical
analysis, no regularity restriction has been imposed on the solution, the
coefficient has not been necessarily deterministic, and the terminal value has
not been necessarily generated by a forward stochastic evolution equation. We
have applied the third temporal semi-discretization to a general stochastic
linear quadratic control problem and established the convergence for a
temporally semi-discrete approximation of the optimal control.


\end{document}